\documentclass[a4paper,10pt]{article}
\usepackage{amsmath, amsthm, amssymb}
\theoremstyle{definition}
\newtheorem{defn}{Definition}[section]

\newtheorem{rem}[defn]{Remark}
\theoremstyle{plain}

\newtheorem{propn}[defn]{Proposition}
\newtheorem{lemma}[defn]{Lemma}
\newtheorem{thm}[defn]{Theorem}
\newtheorem{cor}[defn]{Corollary}
\newtheorem{question}[defn]{Question}
\usepackage{tikz}
\usetikzlibrary{shapes.geometric}
\usepackage{algorithm}
\usepackage{algpseudocode}
\usepackage{graphicx}
\usepackage{hyperref}
\usepackage{enumerate}
\usepackage[paper=a4paper]{geometry}
\usepackage{lscape}

\title{Subsemigroup, ideal and congruence growth of free semigroups}
\author{Alex Bailey\footnote{Mathematical Sciences, University of Southampton, Highfield, Southampton, United Kingdom, SO17 1BJ} \and Martin Finn-Sell\footnote{Mathematische Institut, Georg-August-Universit\"{a}t G\"{o}ttingen, Bunsenstrasse 3-5, 37073, G\"{o}ttingen, Deutschland.} \and Robert Snocken\footnotemark[1]}
\date{Sep 2014}

\begin{document}
\maketitle

\begin{abstract}
Using Rees index, the subsemigroup growth of free semigroups is investigated. Lower and upper bounds for the sequence are given and it is shown to have superexponential growth of strict type $n^n$ for finite free rank greater than $1$. It is also shown that free semigroups have the fastest subsemigroup growth of all finitely generated semigroups. Ideal growth is shown to be exponential with strict type $2^n$ and congruence growth is shown to be at least exponential. In addition we consider the case when the index is fixed and rank increasing, proving that for subsemigroups and ideals this sequence fits a polynomial of degree the index, whereas for congruences this fits an exponential equation of base the index. We use these results to describe an algorithm for computing values of these sequences and give a table of results for low rank and index.
\end{abstract}

\section{Introduction}

The concept of word growth of a finitely generated group has been a central research topic connecting differential geometry, geometric and combinatorial group theory for the past 50 years\footnote{For a history, see the introduction to \cite{Grigorchuk.2013}.} Given a finitely generated group, take the sequence that counts the number of elements of the group of length at most $n$ (with respect to some finite generating set). In 1981, answering a question of Milnor, Gromov proved that a finitely generated group is virtually nilpotent if and only if this sequence has polynomial growth \cite{MR623534}. This powerful result indicated the strong connections between a groups algebraic properties and its asymptotic behaviour.

This result inspired the definition of subgroup growth for a group, see for instance \cite{MR1978431}. Given a finitely generated group, take the sequence that counts the number of index $n$ subgroups of the group. In 1993 Lubotzky, Mann and Segal proved that a finitely generated residually finite group is virtually solvable of finite rank if and only if this sequence has polynomial growth \cite{MR1239055}. This area has now been extended to many other mathematical objects, e.g. representation growth and subring growth \cite{MR2807857}.

In this paper we consider the situation for semigroups. In order to define subsemigroup growth of a semigroup we first need to define index. At first, it is not clear what the correct definition should be, and there have been many different attempts depending on the types of semigroups considered. Ideally you want the notion of index to generalise group index, but perhaps more imporantly, you want finite index subsemigroups to preserve important properties of the semigroup. For example, Grigorchuk gave a definition of index for a subsemigroup which generalised group index and impressively extended Gromov's polynomial growth theorem to cancellative semigroups. However, his definition of index does not preseve even the simplest property of finite generation \cite{MR941053}.

In this paper we choose Rees index, which is defined simply to be the cardinality of the complement. This clearly does not generalise group index, but finite Rees index subsemigroups do preserve a very large number of important properties (e.g. being finitely generated/presented, residually finite, solvable word problem, automatic etc.) See the survey article \cite{arxiv-1307.8259} for more details. The very nature of Rees index makes this an inherently combinatorial problem. For example this paper includes results relating to: binomial coefficients, Stirling numbers (of the first and second kind), Bell numbers, Catalan numbers (and the generalised `Fuss-Catalan' numbers) and Fibonacci numbers.

This work, and the techniques involved were inspired by the rank one situation: finite Rees index subsemigroups of the free monogenic semigroup are well studied in the literature under the name of numerical semigroups \cite{MR2564064,MR2498791,MR2601762,MR2608114,MR2863434,MR3053785}. Increasing interest is being shown in numerical semigroups, with applications arising in commutative algebra and algebraic geometry \cite{MR1357822}. It has recently been proved that the numerical semigroups of genus $n$, that is, the Rees index $n$ subsemigroups of the free monogenic semigroup, have Fibonacci-like growth \cite{MR3053785}, answering in the positive a conjecture of Bras-Amor{\'o}s \cite{MR2377597}. 

Recall the following results from the subgroup growth literature (where group index is used). Every free group $F_r$ with free rank $r \ge 2$ has:
\begin{enumerate}
\item subgroup growth of strict type $n^n$ \cite[Cor 2.2]{MR1978431},
\item subnormal subgroup growth of strict type $2^n=n^{n/log(n)}$ \cite[Cor 2.4]{MR1978431},
\item normal subgroup growth of strict type $n^{log(n)}$ \cite[Cor 2.8]{MR1978431}.
\end{enumerate}

Firstly we recall the definitions from the numerical semigroup literature that we adapt to the higher free rank case, then we prove some basic results about generating sets and how to constuct semigroups of both higher and lower index from a given subsemigroup.

Using a generalisation of the subsemigroup tree of \cite{MR2564064} we then illustrate upper and lower bounds for the number of (Rees) index $n$ subsemigroups of the free semigroups $FS_r$ and conclude that, similar to groups, $FS_r$ has subsemigroup growth of strict type $n^n$ for $r \ge 2$. In Section \ref{sect:poly} we consider the situation when the index is fixed and the free rank varies, in which case we show it fits an exact polynomial of degree the index.

We then consider the question of counting just the ideals, after giving upper and lower bounds, we show that free semigroups of rank greater than $1$ have exponential ideal growth of strict type $2^n=n^{n/log(n)}$, and for fixed index they also satisfy a polynomial of degree the index.

Finally congruence growth is considered, where we count the number of congruences with $n$ classes, and we show this to have at least exponential growth. We conjecture that it is in fact exponential. This is analagous to counting normal subgroups of free groups which have intermediate growth, so in some sense free semigroups have `more' quotients than free groups. We also consider the sequence for a fixed number of classes $n$ as the free rank $r$ increases and show this satisfies an exponential equation of base $n$.

Part of this project was computational. We use some of our results to describe algorithms which were implemented to calculate values of the sequences for low rank and index. The code is available online \cite{code} and tables of the results are presented at the end of the paper as appendices.

\section{Preliminaries}

Let $S$ be a semigroup, the \textbf{Rees index} of a subsemigroup  $T$ of $S$ is defined to be $|S \setminus T|$, and we say $T$ has finite index in $S$ if $|S\setminus T|<\infty$.

Let $X_r=\{g_1,\ldots,g_r\}$ be a finite set of symbols, and $FS_r=X_r^+$ denote the free semigroup of rank $r$ on the set $X_{r}$, that is, all non-empty words over the alphabet $X_r$.

Let $\Lambda \subseteq FS_r$ be a finite (Rees) index subsemigroup of $FS_r$. Following terminology from numerical semigroups, we call $G(\Lambda)=FS_r \setminus \Lambda$ the set of {\bf gaps} of $\Lambda$, and we remark that the index of $\Lambda$ is equal to $|G(\Lambda)|$ (this is usually called the genus). We use $|w|$ to denote the length of any word $w \in FS_r$. We now define a total order on $FS_r$, usually called the \textbf{shortlex order}, where we order first by word length, and then lexicographically. Using the shortlex order we call $f(\Lambda)=\max\{w \mid w \in G(\Lambda)\}$ the {\bf Frobenius} of $\Lambda$, and $m(\Lambda)=\min\{w \mid w \in \Lambda\}$ the {\bf multiplicity} of $\Lambda$.

Given any word $w \in FS_r$ and any $1 \le i \le |w|-1$, let $w_{\text{pre}(i)}$ denote the \textbf{prefix} of $w$ of length $i$, and $w_{\text{suf}(i)}$ denote the \textbf{suffix} of $w$ of length $i$.

\begin{lemma} \label{frob-bound}
For any index $n$ subsemigroup $\Lambda$ of $FS_r$, $|f(\Lambda)| \le 2n-1$.
\end{lemma}

\begin{proof}
Let $\Lambda \subseteq FS_r$ with $|G(\Lambda)|=n$, and assume $|f(\Lambda)|=k \ge 2n$. For every $1 \le i \le n$, either $f_{\text{pre}(i)} \in G(\Lambda)$ or $f_{\text{suf}(k-i)} \in G(\Lambda)$ as otherwise $f \notin G(\Lambda)$. Therefore, along with $f \in G(\Lambda)$, there must be at least $n+1$ distinct words in $G(\Lambda)$ which is a contradiction.
\end{proof}

This immediately gives the following result:

\begin{cor} \label{finitely-many-index-n}
There are only finitely many index $n$ subsemigroups of $FS_r$.
\end{cor}

\begin{proof}
Let $s=r+r^2+\dots+r^{2n-1}$ be the number of words of $FS_{r}$ of length less than $2n$, then there are at most $\binom{s}{n}$ possible choices for the set of gaps.
\end{proof}

Another basic result that will be used frequently is the following:

\begin{lemma}
Every finite index subsemigroup of $FS_r$ has a finite unique minimal generating set.
\end{lemma}

\begin{proof}
Let $\Lambda \subseteq FS_r$ be an index $n$ subsemigroup of $FS_r$. It is clear that $\Lambda$ is finitely generated as every word of length at least $4n$ can be written as a product of two words of length at least $2n$ which are all in $\Lambda$ by Lemma \ref{frob-bound}. Therefore there is a minimal size, say $k \in \mathbb{N}$, for any generating set of $\Lambda$. Take two minimal generating sets $X=\{x_1,\ldots,x_k\}$ and $Y=\{y_1,\ldots,y_k\}$ of $\Lambda$ and assume $X \ne Y$. Let $y_i \in Y \setminus X$, then $y_i=x_{i_1}x_{i_2}\cdots x_{i_l}$ for some $x_{i_1},x_{i_2},\ldots,x_{i_l} \in X$ with $l \ge 2$. Therefore $|y_i| > |x_{i_j}|$ for each $1 \le j \le l$. Now since $Y$ is a generating set, every $x_{i_j}$ is a product of elements from $Y \setminus \{y_i\}$, and so $y_i$ is a product of elements from $Y \setminus \{y_i\}$ contradicting the minimality of $Y$, hence $X=Y$.
\end{proof}

Given any finitely generated semigroup $S$, let $a_n(S)$ denote the number of index $n$ subsemigroups of $S$, and let $s_n(S)=\sum_{i=1}^na_i(S)$ be the partial sums. That $a_n(S)$ is always finite is implied by Corollary \ref{finitely-many-index-n} and the next result which essentially says that free semigroups have the fastest subsemigroup growth.

\begin{propn} \label{free-max-growth}
If $S$ can be generated by $r$ elements, then $a_n(S) \le a_n(FS_r)$ for all $n$.
\end{propn}

\begin{proof}
Let $S=\langle s_1,\dots,s_r\rangle$, then there exists an epimorphism $\phi:FS_r \to S$. Given any $s \in S$ there exists a minimal $\overline{s} \in FS_r$ with respect to the shortlex order such that $\phi(\overline{s})=s$. Let $T$ be an index $n$ subsemigroup of $S$ with $S \setminus T=\{w_1,\dots,w_n\}$ and let $G=\{\overline{w_1},\dots,\overline{w_n}\}$. The result follows if we can show that $\Lambda=FS_r \setminus G$ is a subsemigroup of $FS_r$. Assume that there exists $\overline{w_i} \in G$, and $x,y \in \Lambda$ such that $\overline{w_i}=xy$, then $w_i=\phi(\overline{w_i})=\phi(x)\phi(y)$. Since $T$ is a subsemigroup, either $\phi(x)$ or $\phi(y)$ is in $S \setminus T$. Without loss of generality, assume $\phi(x)=w_j \in S \setminus T$. Since $\phi(\overline{w_j})=\phi(x)$ and $x \notin G$, then $\overline{w_j} < x$ by the minimality of $\overline{w_j}$. However, that implies $\overline{w_j}y < \overline{w_i}$ but $\phi(\overline{w_j}y)=\phi(\overline{w_j})\phi(y)=\phi(x)\phi(y)=w_i$ which contradicts the minimality of $\overline{w_i}$ and so $\Lambda$ is a subsemigroup.
\end{proof}

Recall that given two sequences $f(n),g(n)$, we say that:
\begin{itemize}
\item $f(n)=O(g(n))$ if there exists a constant $C>0$ such that $f(n) \le C \cdot g(n)$ for all large $n$.
\item $f(n) \asymp g(n)$ if $f(n)=O(g(n))$ and $g(n)=O(f(n))$. 
\item $f(n) \sim g(n)$ if $f(n)/g(n) \rightarrow 1$ for large $n$.
\end{itemize}

We say that a semigroup $S$ has {\bf subsemigroup growth of strict type $f(n)$} if
\[
log(s_n(S)) \asymp log(f(n)).
\]
Note that, unless it has a subscript, $log$ is always base $2$.

\section{Minimal generators}
In this section we show how the minimal generators of a finite index subsemigroup of $FS_r$ are connected to its set of gaps. In particular we show the different forms that a generator can take and use this information to outline an algorithm for calculating the set of minimal generators from the set of gaps.

Given a finite index subsemigroup $\Lambda \subseteq FS_r$, let $MG(\Lambda)$ denote the set of minimal generators of $\Lambda$.
\begin{rem}\label{min-gen-rem}
Note that $h \in MG(\Lambda)$ if and only if $h \in \Lambda$ and $\{h_{\text{pre}(i)},h_{\text{suf}(|h|-i)}\} \cap G(\Lambda) \neq \emptyset$ for all $1 \le i \le |h|-1$. This remark will be used ubiquitously without reference throughout this paper.
\end{rem}

Given any index $n$ subsemigroup $\Lambda \subseteq FS_{r}$ we can construct new semigroups in the following ways:
\begin{enumerate}
\item Given $f=f(\Lambda)$, let $\Lambda^f$ denote the set $\Lambda \cup \{f\}$, which is an index $n-1$ subsemigroup as every product $wf$ or $fw$ with $w \in \Lambda$, is larger than $f$ in shortlex order (and hence an element of $\Lambda^f$).
\item Given any $h \in MG(\Lambda)$, let $\Lambda_h$ denote the set $\Lambda \setminus \{h\}$. This is a subsemigroup of index $n+1$, as no pair of elements $x_{1}$,$x_{2} \in \Lambda$ satisfy $x_{1}x_{2}=h$. This is the content of Remark \ref{min-gen-rem}.
\end{enumerate}

Observe that in case (1), $f$ becomes a minimal generator of $\Lambda^f$, and in case (2), $h$ is no longer a minimal generator. In general these constructions will make changes to the minimal generating set. Minimal generators of $\Lambda$ that are no longer minimal generators of $\Lambda^f$ when $f$ is added are said to \textbf{turn off}. Similarly, elements of $\Lambda$ that become new minimal generators of $\Lambda_{h}$ are said to \textbf{turn on}.

The remainder of this section is considering the interplay of these two constructions and their effect on the minimal generating set.

\begin{lemma} \label{min-gen-forms-lemma}
Given a finite index subsemigroup $\Lambda$ of $FS_r$, let $h \in MG(\Lambda)$ and let $Y=MG(\Lambda_h) \setminus (MG(\Lambda)\setminus \lbrace h \rbrace)$ be the set of new generators turned on when we remove $h$ from $\Lambda$. Then $x \in Y$ only if it has one of the following three forms:
\begin{enumerate}
\item $x=hw$ with $w \in MG(\Lambda)$; or \label{form-gw}
\item $x=wh$ with $w \in MG(\Lambda)$; or \label{form-wg}
\item $x=hwh$ with $w \in MG(\Lambda)$.  \label{form-gwg}
\end{enumerate}
\end{lemma}

\begin{proof}
Let $x \in Y$ and note that $G(\Lambda_h)=G(\Lambda) \cup \{h\}$. Then $x \in \Lambda_h$ and $\{x_{\text{pre}(i)},x_{\text{suf}(|x|-i)}\} \cap \left(G(\Lambda) \cup \{h\}\right) \neq \emptyset$ for all $1 \le i \le |x|-1$, but $x \notin MG(\Lambda)$ and so there exists some $x_1,x_2 \in \Lambda$ such that $x=x_1x_2$. Therefore either $x_1=h$ or $x_2=h$. We consider each case separately:
\begin{enumerate}
\item Let $x_1=h$. If $x_2 \in MG(\Lambda)$ then $x$ is of form \ref{form-gw}. So assume $x_2 \notin MG(\Lambda)$ and $x_2=a_1 \cdots a_k$, where $a_1,\ldots,a_k \in MG(\Lambda)$ and $k \ge 2$. Since $x \in MG(\Lambda_h)$ then either $ha_1 \in G(\Lambda) \cup \{h\}$ or $a_2 \cdots a_k \in G(\Lambda) \cup \{h\}$. Now if $ha_1 \in G(\Lambda) \cup \{h\}$, then $ha_1 \in G(\Lambda)$ which is a contradiction as both $h,a_1 \in \Lambda$. Therefore $a_2 \cdots a_k \in G(\Lambda) \cup \{h\}$ which implies $a_2 \cdots a_k=h$. If $k \ge 3$ then we get a contradiction as $h$ is a minimal generator in $\Lambda$. Hence $k=2$, $a_2=h$ and $x$ satisfies form \ref{form-gwg}.
\item Let $x_2=h$. If $x_1 \in MG(\Lambda)$ then it is of form \ref{form-wg}. Otherwise the proof is similar to the previous case and $x$ is of form \ref{form-gwg}. \qedhere
\end{enumerate}
\end{proof}

\begin{lemma} \label{lemma-min-gens-from-gaps}
Given any finite index subsemigroup $\Lambda$ of $FS_r$, we have $h \in MG(\Lambda)$ only if it has one of the following four forms:
\begin{enumerate}
\item $h=g_i$ where $g_i \in X_r$; or \label{form-g}
\item $h=xg_i$ where $x \in G(\Lambda)$, $g_i \in X_r$; or \label{form-xg}
\item $h=g_ix$ where $x \in G(\Lambda)$, $g_i \in X_r$; or \label{form-gx}
\item $h=xg_iy$, where $x,y \in G(\Lambda)$, $g_i \in X_r$. \label{form-xgy}
\end{enumerate}
\end{lemma}

\begin{proof}
We prove this by induction on the index. Let $|G(\Lambda)|=0$, then $x$ must be of form \ref{form-g}. Assume the statement is true for all subsemigroups of index $n$. Given any subsemigroup $\Lambda$ with index $n+1$, we can construct a subsemigroup of index $n$ by considering $\Lambda^f$, where $f=f(\Lambda)$. If $h \in MG(\Lambda^f)$, then by assumption $h$ has the correct form. Suppose that $h \not \in MG(\Lambda^f)$ then $h$ is turned on when we remove $f$ from $\Lambda^f$, so by Lemma \ref{min-gen-forms-lemma} it has one of the forms $wf$, $fw$ or $fwf$ where $w \in MG(\Lambda^f)$. We consider each case separately:

\begin{enumerate}
\item Let $h=fw$ where $w \in MG(\Lambda^f)$, then either $|w|=1$ and $h$ is of form \ref{form-xg}, or $|w| \ge 2$ and $w=g_iw'$ for some $g_i \in X_r$, $w' \in FS_r$. As $f$ is the Frobenius of $\Lambda$, we know that $fg_i \in \Lambda$. Now $h \in MG(\Lambda)$, so we must have $w' \in G(\Lambda)$ and $h$ is of form \ref{form-xgy}.
\item Let $h=wf$ where $w \in MG(\Lambda^f)$, then either $|w|=1$ and $h$ is of form \ref{form-gx}, or similarly to the previous case, $h$ is of form \ref{form-xgy}.
\item Let $h=fwf$, then either $|w|=1$ and $h$ is of form \ref{form-xgy}, or $|h| \ge 2$ and $fwf=(fx_1)(x_2f)$ is not a minimal generator of $\Lambda$. \qedhere
\end{enumerate}
\end{proof}

We can now use Lemma \ref{lemma-min-gens-from-gaps} to outline an algorithm that will calculate the set $MG(\Lambda)$ from the set $G(\Lambda)$. Let $\epsilon$ be the empty word.

\clearpage

\begin{algorithm} \caption{Find minimal generators from set of gaps}\label{alg:findmingen}
\begin{algorithmic}[1]
\Require{$G(\Lambda)$ the set of gaps of a finite index subsemigroup $\Lambda$ of $FS_r$, and $X_r$}
\Procedure{Generators}{$G(\Lambda),X_r$}
\State Gens $\gets \{\}$
\ForAll{$x \in G(\Lambda) \cup \{\epsilon\}$}
\ForAll{$y \in G(\Lambda) \cup \{\epsilon\}$}
\ForAll{$g \in X_r$}
\If{\textsc{MinGen}$(xgy,G(\Lambda))$}
\State Gens $\gets$ Gens $\cup \{xgy\}$
\EndIf
\EndFor
\EndFor
\EndFor
\State \textbf{return} Gens
\EndProcedure
\end{algorithmic}
\end{algorithm}

Where \textsc{MinGen}$(w,G(\Lambda))$ determines whether or not a word $w \in \Lambda$ is a minimal generator of $\Lambda$.

\begin{algorithm} \caption{Check if a word is a minimal generator}\label{alg:checkmingen}
\begin{algorithmic}[1]
\Require{A word $w \in \Lambda$ and the set of gaps $G(\Lambda)$}
\Procedure{MinGen}{$w,G(\Lambda)$}
\State Pass $\gets$ FALSE
\If{$w \notin G(\Lambda)$}
\State Pass $\gets$ TRUE
\State $n \gets$ length$(w)$
\State $i \gets 1$
\While{Pass and $i < n$}
\If{$w_{\text{pre}(i)} \notin G(\Lambda)$ and $w_{\text{suf}(n-i)} \notin G(\Lambda)$}
\State Pass $\gets$ FALSE
\EndIf
\State $i \gets i+1$
\EndWhile
\EndIf
\State \textbf{return} Pass
\EndProcedure
\end{algorithmic}
\end{algorithm}

\section{Subsemigroup growth}\label{sect:uandl}

The sequence $a_n(FS_r)$ has been extensively studied for when $r=1$, this is precisely the number of numerical semigroups of genus $n$. It has recently been proved \cite{MR3053785} that $a_n(FS_1)$ has `Fibonacci like' growth, that is, $a_n(FS_1) \sim K \phi^n$, where $K$ is a constant and $\phi=\frac{1+\sqrt{5}}{2}$ is the golden ration. In this section we give lower and upper bounds for $a_n(FS_r)$ and in particular show that for $r \ge 2$, $a_n(FS_r)$ grows superexponentially in $n$ with strict growth type $n^n$. We imitate the methods used in \cite{MR2564064} by constructing a tree of all finite index subsemigroups of $FS_r$.

\subsection{Subsemigroup tree}
It is clear that every subsemigroup $\Lambda \subseteq FS_r$ of index $n+1$ gives rise to a subsemigroup $\Lambda^f$ of index $n$. Therefore, every index $n$ subsemigroup can be obtained from a subsemigroup of index $n-1$ by removing a minimal generator larger than the Frobenius (using the shortlex order). Given a subsemigroup $\Lambda$ of finite index, any subsemigroup $\Lambda_{h}$ that is obtained by removing a minimal generator $h$ larger than the Frobenius  $f(\Lambda)$ is a descendant. This gives a method for obtaining a tree of all finite index subsemigroups of $FS_r$. 

For example: given the index $1$ subsemigroup $\{a\}^c$ of $FS_2=\langle a,b \rangle$, the minimal generators of $\{a\}^c$ are $\{b,a^2,ab,ba,a^3,aba\}$ and they are all bigger than the Frobenius, the minimal generators of $\{b\}^c$ are $\{a,ab,ba,b^2,bab,b^3\}$ but $a$ is not bigger than the Frobenius and so $\{a,b\}^c$ is a descendant of $\{a\}^c$ but not of $\{b\}^c$.

\medskip

So the beginning of the tree of all index $n$ subsemigroups of $FS_2$ looks like Figure 1. 

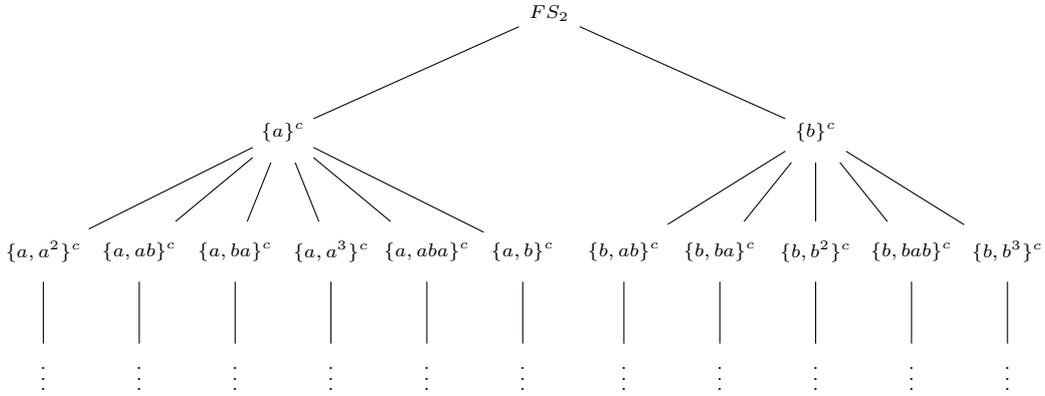
\begin{figure}[ht] \center \label{beg-of-tree}
\begin{tikzpicture}[level distance=4.5em]
\tikzstyle{level 1}=[sibling distance=25em]
\tikzstyle{level 2}=[sibling distance=4.5em]
\scriptsize
\node [rectangle, minimum size=2.8em] (FS2) {$FS_2$}
child {node [rectangle, minimum size=2.8em] (a) {$\{a\}^c$}
	child {node [rectangle, minimum size=2.8em] (a2) {$\{a,a^2\}^c$}
		child {node {$\vdots$}}}
	child {node [rectangle, minimum size=2.8em] (ab) {$\{a,ab\}^c$}
		child {node {$\vdots$}}}
	child {node [rectangle, minimum size=2.8em] (ba) {$\{a,ba\}^c$}
		child {node {$\vdots$}}}
	child {node [rectangle, minimum size=2.8em] (a3) {$\{a,a^3\}^c$}
		child {node {$\vdots$}}}	
	child {node [rectangle, minimum size=2.8em] (aba) {$\{a,aba\}^c$}
		child {node {$\vdots$}}}
		child {node [rectangle, minimum size=2.8em] (a2) {$\{a,b\}^c$}
	child {node {$\vdots$}}}
}
child {node [rectangle,minimum size=2.8em] (b) {$\{b\}^c$}
	child {node [rectangle, minimum size=2.8em] (ab) {$\{b,ab\}^c$}
		child {node {$\vdots$}}}
	child {node [rectangle, minimum size=2.8em] (ba) {$\{b,ba\}^c$}
		child {node {$\vdots$}}}
	child {node [rectangle, minimum size=2.8em] (b2) {$\{b,b^2\}^c$}
		child {node {$\vdots$}}}
	child {node [rectangle, minimum size=2.8em] (bab) {$\{b,bab\}^c$}
		child {node {$\vdots$}}}	
	child {node [rectangle, minimum size=2.8em] (b3) {$\{b,b^3\}^c$}
		child {node {$\vdots$}}}
};
\end{tikzpicture}
\caption{Subsemigroup tree of $FS_2$}
\end{figure}

It is clear that $a_n(FS_r)$ is the number of nodes on the $n^\text{th}$ level of this tree.

\medskip

Inspired by the numerical semigroup situation \cite{MR2498791,MR2564064}, we say that a finite index subsemigroup $\Lambda \subseteq FS_r$ is {\bf ordinary} if $f(\Lambda) < m(\Lambda)$ in the shortlex order, that is, all the gaps are `at the beginning'. For a given rank $r$, clearly there is one ordinary subsemigroup for each index $n$ which we denote by $O_r(n)$.

\begin{lemma} \label{num-ord-desc}
In this tree the subsemigroup $O_r(n)$ has $(r-1)n^2+(2r-1)n+r$ descendants.
\end{lemma}

\begin{proof}Let $O_r(n)$ be the ordinary index $n$ subsemigroup of $FS_r$ and let $k=|f(O_r(n))|$. Figure 2 
represents the set of words of $G(O_r(n))$ within the set of all words in $FS_r$:

\begin{figure}[ht] \center \label{fig-ordinary-gaps}
\begin{tikzpicture}[scale=0.36]

\draw (5,0) -- (0,-15) -- (10,-15) -- (5,0);
\draw (0,-15) -- (-1/3,-16) -- (31/3,-16) -- (10,-15);
\draw (-1/3,-16) -- (-2/3,-17) -- (32/3,-17) -- (31/3,-16);
\draw (11,-18) -- (32/3,-17);
\draw (-1,-18) -- (-2/3,-17);  
\draw (2, -9) -- (5,-9);
\draw (5,-9) -- (5,-8);
\draw (5,-8) -- (23/3,-8);
\draw[dashed] (7/3,-8) -- (5,-8);

\draw[arrows=->,line width=.4pt] (3.75,-8.5) --(4.9,-8.5);
\draw[arrows=->,line width=.4pt] (3.25,-8.5) --(7/3,-8.5);
\draw[arrows=->,line width=.4pt] (1,-5) -- (1,-8.75);
\draw[arrows=->,line width=.4pt] (1,-4) -- (1,-0.25);

{\tiny
\node at (5,-5) {$G(O_r(n))$} ;
\node at (3.5,-8.5) {$i$} ;
\node at (5,-12) {$(1.)$};
\node at (-1,-15.5) {$2k$};
\node at (-2,-16.5) {$2k+1$};
\node at (5,-15.5) {$(2.)$};
\node at (5,-16.5) {$(3.)$};
\node at (1, -4.5) {$k$};
}

\end{tikzpicture}
\caption{Set of gaps $G(O_r(n))$ within $FS_r$.}
\end{figure}
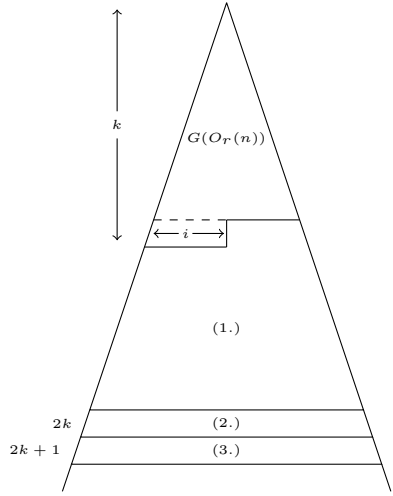

Let $i \ge 1$ be the number of words of length $k$ that belong to the set of gaps $G(O_{r}(n))$. Observe that $n=(\sum_{j=1}^{k-1} r^j)+i$ where $1 \le i \le r^k$.

We now consider, in three separate cases, which of the words $w \in O_r(n)$ are minimal generators. These cases correspond to the numbered regions in Figure 2. 
\begin{enumerate}
\item Let $|w| \le 2k-1$. Then $w \in MG(O_r(n))$ if and only if $w \in O_r(n)$ as $w$ cannot be written as a product of two words in $O_r(n)$. Hence there are $(\sum_{j=k}^{2k-1}r^j)-i$ such $w$ that are minimal generators.
\item Let $|w|=2k$. Then $w \in MG(O_r(n))$ if and only if $w=w_kw_k'$ where $|w_k|=|w_k'|=k$ and at least one of $w_k$, $w_k'$ are in $G(O_r(n))$. Being careful not to double count the cases where $w_k=w_k'$ we have that there are $\sum_{j=0}^{i-1}(2(r^k-j)-1)=2ir^k-i^2$ such $w \in O_r(n)$.
\item Let $|w| \ge 2k+1$. Clearly no word of length at least $2k+2$ is a minimal generator as every such word is a product of two words of length at least $k+1$, all of which are in $O_r(n)$. This leaves the words of length $2k+1$, which are minimal generators if and only if $w_{\text{pre}(k)} \in G(\Lambda)$ and $w_{\text{suf}(k)} \in G(\Lambda)$, with the middle letter of $w$ being any of the $r$ generators of $FS_r$. There are therefore $ri^2$ such words.
\end{enumerate}
Hence
\begin{align*}
|MG(O_r(n))|&=\left(\left(\sum_{j=k}^{2k-1}r^j\right)-i\right)+(2ir^k-i^2)+(ri^2)
\end{align*}
After some manipulation, we deduce that:
\begin{align*}
|MG(O_r(n))|&=(r-1)\left(\left(\sum_{j=1}^{k-1}r^j\right)+i\right)^2+(2r-1)\left(\left(\sum_{j=1}^{k-1}r^j\right)+i\right)+r \\
&=(r-1)n^2+(2r-1)n+r.
\end{align*}
Since every word in $O_r(n)$ is bigger than the Frobenius, every minimal generator gives rise to a descendant and the result follows.
\end{proof}

\subsection{Lower bound}

We now describe how to construct a lower bound for the sequence $a_n(FS_r)$ by constructing a subtree of the subsemigroup tree. 

Following \cite{MR2498791}, we construct a subtree of the subsemigroup tree using the ordinary subsemigroups. We describe it as folows: beginning with the single vertex representing $FS_{r}$ we attach to this root the subsemigroups of index $1$ obtained by removing each generator. We now proceed inductively. Consider a node (that represents a subsemigroup $\Lambda$):
\begin{enumerate}
\item If the subsemigroup $\Lambda$ is ordinary then attach, for every $h \in MG(\Lambda)$, a node for $\Lambda_h$.
\item If the subsemigroup $\Lambda$ is not ordinary, then it is obtained by removing minimal generators from some ordinary subsemigroup of a lower index. Add nodes for each of the minimal generators $h$ of $\Lambda$ that are bigger than the Frobenius $f(\Lambda)$ and are also minimal generators of this ordinary subsemigroup.
\end{enumerate}

This is a subtree of the subsemigroup tree that has one infinite branch consisting of ordinary subsemigroups of $FS_{r}$. The number of nodes on the $n^{\text{th}}$ level of the tree is a lower bound for $a_n(FS_r)$. Recalling the result from Lemma \ref{num-ord-desc}, if we let $p(n,r)=(r-1)n^2+(2r-1)n+r$ be the number of descendants of $O_r(n)$, then Figure 3 
illustrates the beginning of this subtree for $FS_{2}$.

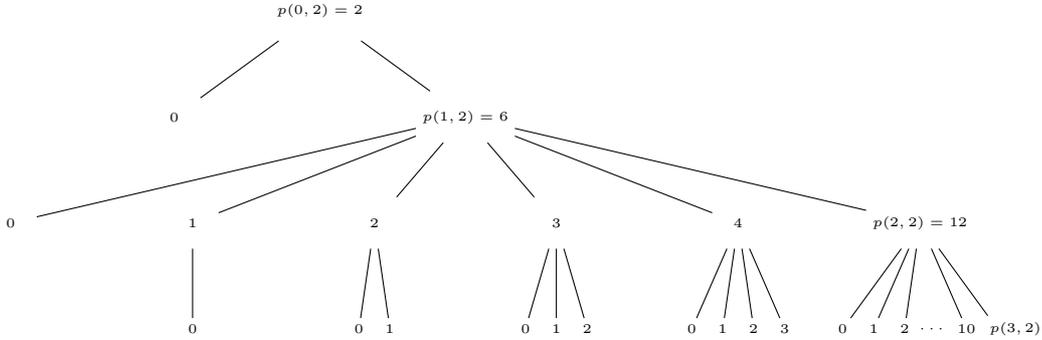
\begin{figure}[ht] \center \label{subtree}
\begin{tikzpicture}[level distance=4em]
\tikzstyle{level 1}=[sibling distance=16em]
\tikzstyle{level 2}=[sibling distance=10em]
\tikzstyle{level 3}=[sibling distance=1.7em]
\tiny
\node (2em,0) [circle] (z){$p(0,2)=2$}
child {node [rectangle, minimum size=2.8em] (a) {$0$}}
child {node [rectangle,minimum size=2.8em] (o1) {$p(1,2)=6$}
	child {node [rectangle, minimum size=2.8em] (o1a) {$0$}}
	child {node [rectangle, minimum size=2.8em] (o1b) {$1$}
		child {node {$0$}}}
	child {node [rectangle, minimum size=2.8em] (o1c) {$2$}
		child {node {$0$}}
		child {node {$1$}}}
	child {node [rectangle, minimum size=2.8em] (o1d) {$3$}
		child {node {$0$}}
		child {node {$1$}}
		child {node {$2$}}}
	child {node [rectangle, minimum size=2.8em] (o1e) {$4$}
		child {node {$0$}}
		child {node {$1$}}
		child {node {$2$}}
		child {node {$3$}}}
	child {node [rectangle, minimum size=2.8em] (o1f) {$p(2,2)=12$}
		child {node {$0$}}
		child {node {$1$}}
		child {node (o1g) {$2$}}
		child[white] {node {}}
		child {node (o1h) {$10$}}
		child {node {$\qquad p(3,2)$}}}
};
\path (o1g) -- (o1h) node (dts) [midway] {$\cdots$};
\end{tikzpicture}
\caption{Subtree of subsemigroup tree for $FS_2$}
\end{figure}

Observe that every node belongs to a branch starting from some ordinary subsemigroup $O_r(n)$, and corresponds to choosing a subset of $MG(O_r(n))$. Therefore it is clear that the number of nodes on the $n^{\text{th}}$ level of this tree is

\begin{equation*}
L(n,r):=\sum_{i=0}^{J(n,r)} {p(n-i,r)-1 \choose i}.
\end{equation*}

where
\[
J(n,r)=
\begin{cases}
\hfil n/2 & \text{for } r=1 \\
\left\lfloor{(n-1)+\frac{(2r-1)-\sqrt{(r-1)n+(2r-1))}}{r-1}}\right\rfloor & \text{for } r>1
\end{cases}
\]
where $J(n,r)$ is obtained from the inequality $p(n-i,r)-1 \ge i$. Hence we have proved the following Theorem:

\begin{thm}
$a_n(FS_r) \ge L(n,r)$ for all $r,n$.
\end{thm}

Note that $L(n,1)=F_{n+1}$ the Fibonacci numbers, which is a good lower bound for the numerical semigroup case, since $a_n(FS_1) \sim C \cdot L(n,1)$ for some constant $C$ \cite{MR3053785}. It seems reasonable to conjecture that for each $r \ge 1$, $a_n(FS_r) \sim C_r \cdot L(n,r)$ for some constant $C_r$.

\medskip

The proof in \cite{MR3053785} that $a_n(FS_1) \sim C \cdot L(n,1)$ relied on the fact that almost all numerical semigroups satisfy $f<3m$, where $m$ is the multiplicity and $f$ is the frobenius. It would be of interest to know whether this proof can be extended to the higher rank case, that is, do almost all finite index subsemigroups $\Lambda \subseteq FS_r$ satisfy $|f(\Lambda)|<3\cdot |m(\Lambda)|$?

\medskip

\begin{thm} \label{growth-below}
For $r \ge 2$, $log(n^n)=O(log(s_n(FS_r)))$.
\end{thm}

\begin{proof}
First note that $n/2 \le J(n,r)$ so by considering the term $i=n/2$ we have
\[
s_n(FS_r) \ge a_n(FS_r) \ge L(n,r) \ge {p(n-n/2,r)-1 \choose n/2}.
\]
Note, we can continuously extend the binomial coeffecients using Gamma functions, so we need not worry whether $n$ is even. When $r \ge 2$, $p(n,r) \ge n^2+1$ and so
\[
s_n(FS_r) \ge {(n/2)^2 \choose n/2} \ge \left(\frac{(n/2)^2}{n/2}\right)^{n/2}=(n/2)^{n/2}.
\]
Thus $log(s_n(FS_r)) \ge \frac{n}{2}(log(n)-log(2))$ and so $nlog(n)=O(log(s_n(FS_r)))$.
\end{proof}

\subsection{Upper bound}
We now construct an upper bound for the sequence $a_n(FS_r)$. In order to do this we show that, similar to the situation for numerical semigroups (see \cite{MR2498791}), ordinary subsemigroups have the maximum number of descendants. To show this we require the following Proposition:

\begin{propn} \label{add-f-rem-m}
Let $\Lambda$ be a finite index non-ordinary subsemigroup of $FS_r$, then $|MG(\Lambda)| \le |MG(\Lambda^f_m)|$.
\end{propn}

\begin{proof}
We prove the statement by showing that every minimal generator of $\Lambda$ turned off by adding $f$ gives rise to a unique new minimal generator of $\Lambda^f_m$ turned on by removing $m$ from $\Lambda^f$. First note that the minimal generators of $\Lambda$ that are turned off by adding $f$ are precisely the new minimal generators of $\Lambda$ turned on by removing the minimal generator $f$ from $\Lambda^f$. By Lemma \ref{min-gen-forms-lemma}, such a minimal generator $x \in D=MG(\Lambda) \setminus (MG(\Lambda^f)\setminus \lbrace f \rbrace)$ has one of three possible forms $fw, wf$ or $fwf$ where $w \in MG(\Lambda^f)$. Note that $x$ could have more than one of these forms. We now partition $D$ in to nine distinct (possibly empty) sets.
\begin{align*}
D_1:&=\{fwf \in D \mid w \in MG(\Lambda^f)\} \\
D_2:&=\{fyf \in D \mid yf \text{ or } fy \in MG(\Lambda^f), my \notin \Lambda^f\} \\
D_3:&=\{fyf \in D \mid yf \text{ or } fy \in MG(\Lambda^f), my \in \Lambda^f, ym \notin \Lambda^f\} \\
D_4:&=\{fyf \in D \mid yf \text{ or } fy \in MG(\Lambda^f), my, ym \in \Lambda^f\} \\
D_5:&=\{fw \in D \mid w \in MG(\Lambda^f), m \neq w \neq f, fw \notin D_1 \cup D_2 \cup D_3 \cup D_4\} \\
D_6:&=\{wf \in D \mid w \in MG(\Lambda^f), m \neq w \neq f, wf \notin D_1 \cup D_2 \cup D_3 \cup D_4 \cup D_5\} \\
D_7:&=\{fm \in D \mid fm \notin D_6 \} \\
D_8:&=\{mf \in D \mid mf \notin D_5 \cup D_7\} \\
D_9:&=\{f^2 \in D\}
\end{align*}
It is a straightforward matter to check that $D_1,\dots,D_9$ are indeed disjoint and that their union is $D$ (in particular  $D_2$ is disjoint from $D_1$ as $fy$ or $yf \in MG(\Lambda^f)$ implies $y \in G(\Lambda)$, so $y \notin MG(\Lambda)$). Now for each $x \in D$, we are going to assign some $x \in C=MG(\Lambda^f_m) \setminus (MG(\Lambda^f)\setminus \lbrace m \rbrace)$. All the proofs here use Remark \ref{min-gen-rem}.
\begin{enumerate}
\item Let $x=fwf \in D_1$. We intend to show that $mwm \in C$. Firstly, $|w|=1$ as otherwise $w=x_1x_2$ for some $x_1,x_2 \in FS_r$ and $x=(fx_1)(x_2f) \notin MG(\Lambda)$ as $fx_1,x_2f \in \Lambda$. Also, since $m$ and everything smaller than $m$ is in $G(\Lambda^f_m)$, we have $mwm_{\text{pre}(i)} \in G(\Lambda^f_m)$ for all $1 \le i \le |m|$ and $mwm_{\text{suf}(|mwm|-i)} \in G(\Lambda^f_m)$ for all $|m|+1=|mw| \le i \le |mwm|-1$. Now since $w,m \in \Lambda^f$ we see that $mwm \in \Lambda^f$, and $mwm \in MG(\Lambda^f_m)$ but $(mw)(m) \notin MG(\Lambda^f)$, so $mwm \in C$.
\item Let $x=fyf \in D_2$. Firstly, as before, $|y|=1$ as otherwise $fyf \notin MG(\Lambda)$. Secondly, since $my \notin \Lambda^f$, then $my \notin \Lambda^f_m$. Now, since $m$ and everything smaller than $m$ is not in $\Lambda^f_m$ we have $myf_{\text{pre}(i)} \in G(\Lambda^f_m)$ for all $1 \le i \le |m|$, and $myf_{\text{pre}(|myf|-(|m|+1))}=my \in G(\Lambda^f_m)$ and since $x \in MG(\Lambda)$ and $fy \in \Lambda$, we also have $myf_{\text{suf}(|myf|-i)} \in G(\Lambda^f_m)$ for all $|m|+2 \le i \le |myf|-1$. Therefore since $myf \in \Lambda^f_m$ we have $myf \in MG(\Lambda^f_m)$ but $(m)(yf) \notin MG(\Lambda^f)$, and so $myf \in C$.
\item Let $x=fyf \in D_3$. Similarly to the previous case, $fym \in C$.
\item Let $x=fyf \in D_4$. Again, $|y|=1$ as otherwise $fyf \notin MG(\Lambda)$. Since $m$ and everything smaller than $m$ is not in $\Lambda^f_m$, we have $mym_{\text{pre}(i)} \in G(\Lambda^f_m)$ for all $1 \le i \le |m|$ and $mym_{\text{suf}(|mym|-i)} \in G(\Lambda^f_m)$ for all $|m|+1=|my| \le i \le |mym|-1$. Since $m$ and $my \in \Lambda^f$ we see $m \neq mym \in \Lambda^f_m$ and $mym \in MG(\Lambda^f_m)$ but $(my)(m) \notin MG(\Lambda^f)$, and so $mym \in C$.
\item Let $x=fw \in D_5$. Firstly, since $fw \notin D_1 \cup D_2 \cup D_3 \cup D_4$, then $w$ does not have $f$ as a proper suffix. Secondly, since $m$ and $w \in \Lambda^f$, then $mw \notin MG(\Lambda^f)$. Now since $fw \in MG(\Lambda)$ and anything bigger than $f$ is in $\Lambda^f_m$ and $w$ does not have $f$ as a proper suffix, we have $fw_{\text{suf}(|fw|-i)} \in G(\Lambda^f_m)$ for all $|f|+1 \le i \le |fw|-1$. This implies $mw_{\text{suf}(|mw|-i)} \in G(\Lambda^f_m)$ for all $|m|+1 \le i \le |mw|-1$ and since $m$ and everything smaller than $m$ is in $G(\Lambda^f_m)$ we also have $mw_{\text{pre}(i)} \in G(\Lambda^f_m)$ for all $1 \le i \le |m|$. Finally, since $m,mw \in \Lambda^f$, then $m \neq mw \in \Lambda^f_m$ and $mw \in MG(\Lambda^f_m)$, so $mw \in C$.
\item Let $x=wf \in D_6$. Similarly to the previous case, this implies $wm \in C$.
\item Let $x=fm \in D_7$. Firstly, recall that since $fm \in MG(\Lambda)$ we have $fm \in \Lambda$ and $\{fm_{\text{pre}(i)},fm_{\text{suf}(|fm|-i}\} \cap G(\Lambda) \neq \emptyset$ for all $1 \le i \le |fm|-1$. We wish to show that $fm$ is also in $MG(\Lambda^f_m)$. It is enough to check the cases when $f$ is a prefix or suffix of $fm$. Obviously $f$ is a prefix, but in that case $m \in G(\Lambda^f_m)$. Assume $fm$ has $f$ as a suffix, that is $fm=wf$ for some $w \in FS_r$ not equal to $f$. Then $w \not\in MG(\Lambda^f)$ as $fm \not \in D_6$, but $\vert w \vert = \vert m \vert$ implies that $w \in MG(\Lambda^f)$. Hence, $fm$ does not have $f$ as a suffix. Observe also that $fm \in \Lambda^f_m$. Finally, $(f)(m) \notin MG(\Lambda^f)$ and so $fm \in C$.
\item Let $x=mf \in D_8$. Similarly to $x \in D_7$, we have $mf \in C$.
\item Let $x=f^2$. It is clear that $m^2$ is always a minimal generator of $\Lambda^f_m$ as firstly, $m \neq m^2 \in \Lambda^f_m$ and secondly, $m$ and everything smaller than $m$ is in $G(\Lambda^f_m)$ and so $m^2_{\text{pre}(i)} \in G(\Lambda^f_m)$ for all $1 \le i \le |m|$ and $m^2_{\text{suf}(|m^2|-i)} \in G(\Lambda^f_m)$ for all $|m| \le i \le |m^2|-1$. Clearly also $m^2=mm \notin MG(\Lambda^f)$ and so $m^2 \in C$.
\end{enumerate}
Hence we can construct a function
\begin{align*}
\Phi:D &\to C \\
x &\mapsto \begin{cases}
mwm &\text{ if $x=fwf \in D_1$,} \\
myf &\text{ if $x=fyf \in D_2$,} \\
fym &\text{ if $x=fyf \in D_3$,} \\
mym &\text{ if $x=fyf \in D_4$,} \\
mw &\text{ if $x=fw \in D_5$,} \\
wm &\text{ if $x=wf \in D_6$,} \\
fm &\text{ if $x=fm \in D_7$,} \\
mf &\text{ if $x=mf \in D_8$,} \\
m^2 &\text{ if $x=f^2 \in D_9$.} \\
\end{cases}
\end{align*}
Since $D$ is the disjoint union of $D_1,\dots,D_9$, this function is well-defined. Since $FS_r$ is cancellative it is clear that if $x_1,x_2 \in D_i$ for some $1 \le i \le 9$, then $\Phi(x_1)=\Phi(x_2)$ implies $x_1=x_2$. Therefore to show that $\Phi$ is injective, it is enough to show that $\Phi(x_1)$ can never equal $\Phi(x_2)$ whenever $x_1 \in D_i$, $x_2 \in D_j$ and $i \neq j$. There are 36 different cases to check, which are straightforward but tedious. We consider two of the cases and leave the rest to the reader, as they are either trivial or identical in nature to the ones presented.
\begin{enumerate}
\item Let $x_1 \in D_9$ and $x_2 \in D_1$, if $\Phi(x_1)=\Phi(x_2)$ then $m^2=mwm$ which is an immediate contradiction as $|w| \ge 1$.
\item Let $x_{1} \in D_{6}$ and $x_{2} \in D_{5}$. If $\Phi(x_{1})=\Phi(x_{2})$ then $wm=mw'$, hence $w=mz$ for some $z \in FS_r$. This implies $wf=m(zf)\notin MG(\Lambda)$ since $m,zf \in \Lambda$ which is a contradiction.
\end{enumerate}

Therefore $\Phi$ is injective and we can extend it to an injective function
\begin{align*}
\Psi:MG(\Lambda) &\to MG(\Lambda^f_m) \\
x & \mapsto \begin{cases}
\Phi(x) & \text{if } x \in D \\
x & \text{if } x \in MG(\Lambda) \cap MG(\Lambda^f), x\not = m\\
f & \text{if } x=m
\end{cases}
\end{align*}
and so $|MG(\Lambda)| \le |MG(\Lambda^f_m)|$.
\end{proof}

\begin{cor} \label{ord-max-desc}
For a fixed index, ordinary subsemigroups of $FS_r$ have the maximum number of descendants in the subsemigroup tree.
\end{cor}

\begin{proof}
Firstly note that ordinary subsemigroups have the maximum number of minimal generators. In fact, given any non-ordinary finite index subsemigroup $\Lambda \subseteq FS_r$, by Proposition \ref{add-f-rem-m}, $\Lambda^f_m$ has no less minimal generators than $\Lambda$ so apply Proposition \ref{add-f-rem-m} to $\Lambda$ finitely many times until it is ordinary. Since every minimal generator of an ordinary subsemigroup is bigger than the Frobenius, each gives rise to a descendant and the result follows.
\end{proof}

Now we can prove the following Theorem using the information from Lemma \ref{num-ord-desc} 

\begin{thm}
For $r \ge 2$, $a_n(FS_r) \le (r-1)^n(n+1)(n!)^2$.
\end{thm}

\begin{proof}
By Corollary \ref{ord-max-desc}, the ordinary subsemigroups have the maximum number of descendants, which by Lemma \ref{num-ord-desc}, is $(r-1)n^2+(2r-1)n+r$. So let us assume every subsemigroup has this number of descendants to construct an upper bound. Then
\begin{align*}
a_n(FS_r)&\ge\prod_{k=0}^{n-1}\left((r-1)k^2+(2r-1)k+r)\right) \\
&=\prod_{k=0}^{n-1}\left((r-1)(k+1)^2+(k+1)\right) \\
&=\prod_{k=1}^n\left((r-1)k^2+k\right) \le \prod_{k=1}^n\left((r-1)k^2+(r-1)k\right) \\
&=(r-1)^n\prod_{k=1}^n(k(k+1)) \\
&=(r-1)^n(n+1)(n!)^2. \qedhere
\end{align*}
\end{proof}

\begin{thm} \label{growth-above}
For $r \ge 2$, $log(s_n(FS_r))=O(log(n^n))$.
\end{thm}

\begin{proof}
Let $U(n,r)=(r-1)^n(n+1)(n!)^2$. Since $a_n(FS_r) \le U(n,r)$ and since $U(n,r)$ is non-decreasing, we have $s_n(FS_r) \le n \cdot U(n,r)$. Since $n!<n^n$, we have
\[
log(s_n(FS_r)) \le log(n)+nlog(r-1)+log(n+1)+2nlog(n)=O(nlog(n)). \qedhere
\]
\end{proof}

\begin{cor} \label{subsgp-growth-strict-type}
For $r \ge 2$, $FS_r$ has subsemigroup growth of strict type $n^n$.
\end{cor}

\begin{proof}
By Theorems \ref{growth-below} and \ref{growth-above}.
\end{proof}

\section{Subsemigroup growth for a fixed index}\label{sect:poly}
In this section we consider the growth of $a_n(FS_r)$ when $n$ is fixed and the rank $r$ varies. In this case the sequence fits an explicit polynomial. In order to prove this, we need some preliminary remarks. 

Given any word $w=g_{\alpha(1)}\dots g_{\alpha(m)} \in FS_r$, with $g_{\alpha(i)} \in X_r$, and any permutation $\sigma \in \text{Sym}(X_r)$, let $\sigma(w)$ denote the word $\sigma(g_{\alpha(1)})\dots\sigma(g_{\alpha(m)})$. Given $\Lambda$ a finite index subsemigroup of $FS_r$ with $G(\Lambda)=\{w_1,\dots,w_n\}$, then let $\sigma(G(\Lambda))$ denote the set $\{\sigma(w_1),\sigma(w_2),\dots,\sigma(w_n)\}$. It is clear that $FS_r \setminus \sigma(G(\Lambda))$ is also a finite index subsemgroup of $FS_r$ isomorphic to $\Lambda$.

For any index $n$ subsemigroup $\Lambda \subseteq FS_r$ we wish to think of the set of gaps of $\Lambda$ as a `pattern' by forgetting the labels of the generators of $FS_r$. To make this idea precise, let $Fin$ be the set of all finite index subsemigroups of finite rank free semigroups.

We now define an equivalence relation $\sim$ on $Fin$. Let $\Lambda_1, \Lambda_2 \in Fin$ be two finite index subsemigroups with $\Lambda_1 \subseteq FS_p$ and $\Lambda_2 \subseteq FS_q$ say. We say $\Lambda_1 \sim \Lambda_2$ if there exists $\sigma \in \text{Sym}(X_{\max\{p,q\}})$ such that $G(\Lambda_1)=\sigma(G(\Lambda_2))$. If $\Lambda_1 \sim \Lambda_2$ then we say that $\Lambda_1$ and $\Lambda_2$ have the same \textbf{gap pattern}. 

Let $w \in FS_r$. Define the {\bf support} of $w$ to be the set of minimal generators that make up $w$. That is, if $w=g_{\alpha(1)}g_{\alpha(2)}\dots g_{\alpha(m)}$ where $g_{\alpha(i)} \in X_r$, then supp$(w):=\{g_{\alpha(1)},\dots,g_{\alpha(m)}\}$. Given $\Lambda \in Fin$, with $G(\Lambda)=\{w_1,w_2,\dots,w_n\}$, we define supp$(G(\Lambda))$ to be $\bigcup_{i=1}^n\text{supp}(w_i)$. 

Let $\Lambda \in Fin$ with $\vert \text{supp}(G(\Lambda)) \vert=k$. Then $\Lambda$ has a minimal representative $\Lambda_{min} \subseteq FS_k$ of the $\sim$-class of $\Lambda$ so that $\text{supp}(\Lambda_{min})=X_k$ and the elements of the support `first appear in order'. More formally, let $g_j=\max\{g_i \mid g_i \in \text{supp}(\Lambda)\}$ be the largest generator in the gaps of $\Lambda$ using the standard lexicographical order. If we let $y=w_1\dots w_n$ be the concatenation of the set of gaps of $\Lambda$, then we can totally order the $\sigma \in \text{Sym}(X_j)$ using the lexicographcial order on $\sigma(y)$. Let $\sigma_{min}$ be the smallest permutation with respect to this total order, then $\Lambda_{min}$ is the complement of $\sigma_{min}(G(\Lambda))$ in $FS_k$.

Now, for a fixed $\Lambda \in Fin$ let the {\bf orbit} of $\Lambda$ be defined as Orb$(\Lambda):=\{\sigma(\Lambda) \mid \sigma \in \text{Sym$($supp}(\Lambda))\}$. This is the set of all elements of $Fin$ with the same gap pattern and the same support as $\Lambda$. By the orbit-stabilizer theorem, $|\text{Orb}(\Lambda)|$ is a divisor of $|\text{Sym}(\text{supp}(\Lambda))|=|\text{supp}(\Lambda)|!$.

\smallskip

For each $n \ge 1$, let $Z(n)$ be the set of minimal representatives of $\sim$-classes of index $n$ subsemigroups. For each $k \ge 1$, let $Z(n,k):=\{P \in Z(n) \mid |\text{supp}(P)|=k\}$ and for each $i \mid k!$, let $Z(n,k,i):=\{P \in Z(n,k) \mid |\text{Orb}(P)|=i\}$. Note, we have that
\begin{equation}\label{eqn-1}
Z(n,k)=\bigsqcup_{i \mid k!}Z(n,k,i).
\end{equation}
Observe that each of the sets $Z(n,k)$ are finite. In fact, as each $P$ is minimal $|Z(n,k)|$ is no bigger than the number of index $n$ subsemigroups of $FS_k$, which by Corollary \ref{finitely-many-index-n}, is finite.

\smallskip

Now given any $r$ and $n$, we wish to determine the number of index $n$ subsemigroups of $FS_r$. For each $P \in Z(n,k)$ there are $k$ possible generators we could choose from $r$ for the support of $P$, and there are $|\text{Orb}(P)|$ different subsemigroups of $FS_r$ with the same gap pattern and the same support as $P$. Therefore, there are $|\text{Orb}(P)|\cdot\binom{r}{k}$ index $n$ subsemigroups of $FS_r$ with the same gap pattern as $P$, and we have the following equation:

\begin{align*}
a_n(FS_r)&=\sum_{k=1}^r \sum_{P \in Z(n,k)}|\text{Orb}(P)|\cdot\binom{r}{k}.
\end{align*}

\;

\begin{lemma} \label{no-z-n-n+}
$Z(n,k)=\emptyset$ for all $k >n$.
\end{lemma}

\begin{proof}
This is equivalent to proving that every index $n$ subsemigroup $\Lambda$ has $|\text{supp}(G(\Lambda))|\le n$. We prove by induction on the index. Let $n=1$ and given any index $1$ subsemigroup $\Lambda$ with $G(\Lambda)=\{w_1\}$, then $w_1$ must be in $X_r$ and so $|\text{supp}(\Lambda)|=1$. Now assume the statement is true for $n \ge 1$. Given any subsemigroup with $|G(\Lambda)|=n+1$, let $f=f(\Lambda)$ and consider the semigroup $\Lambda^f$ which has $|G(\Lambda^f)|=n$. If $G(\Lambda^f)=\{w_1,\dots,w_n\}$ then by the assumption $|\text{supp}(G(\Lambda^f))| \le n$. Since $f$ is a minimal generator of $\Lambda^f$, by Lemma \ref{lemma-min-gens-from-gaps}, $f$ is of the form $w_ig$, $gw_i$ or $w_igw_j$ for some $w_i,w_j \in G(\Lambda^f)$, $g \in X_r$ and therefore supp$(G(\Lambda^f))=\text{supp}(G(\Lambda)) \cup \text{supp}(g)$ and $|\text{supp}(\Lambda)| \le n+1$.
\end{proof}

Therefore we can refine our range of summation slightly to get:
\begin{align*}
a_n(FS_r)&=\sum_{k=1}^n \sum_{P \in Z(n,k)}|\text{Orb}(P)|\cdot\binom{r}{k} \\
&=\sum_{k=1}^n\sum_{i \mid k!}\sum_{P \in Z(n,k,i)} i\ \binom{r}{k} \mbox{\qquad by (\ref{eqn-1})}\\
&=\sum_{k=1}^n\sum_{i \mid k!}|Z(n,k,i)|\cdot i \ \binom{r}{k} \\
&=\sum_{k=1}^n\sum_{i \mid k!}|Z(n,k,i)|\frac{i}{k!} \prod_{j=0}^{k-1}(r-j)
\end{align*}
So if we let
\[
c(n,k):=\sum_{i \mid k!}|Z(n,k,i)|\cdot i
\]
and $s(n,k)$ be the (signed) Stirling numbers of the first kind, then
\begin{align*}
a_n(FS_r)&=\sum_{k=1}^n\frac{c(n,k)}{k!}\sum_{j=0}^ns(n,j)r^j \\
&=\sum_{k=1}^n\left(\left(\sum_{j=k}^n\frac{c(n,j)}{j!}s(n,k)\right) r^k \right)
\end{align*}
We remark that the ordinary subsemigroup $O_n(n)$ has $|\text{supp}(O_n(n))|=n$ and so $c(n,n) \neq 0$. Therefore we have proved:

\begin{thm} \label{poly-conj}
$a_n(FS_r)$ is a polynomial in $r$ of degree $n$ with no constant term.
\end{thm}

We now use this result to describe an algorithm (making use of the previous algorithms) for inductively computing the sets $Z(n,k,i)$ and therefore the values $a_n(FS_r)$. See Algorithm \ref{alg:polynomial}.

\begin{algorithm} \caption{Find the sets $Z(n,k,i)$ for all $i$, from the sets $Z(n-1,k,i)$ and $Z(n-1,k-1,i)$}\label{alg:polynomial}
\begin{algorithmic}[1]
\Require{Index $n$, support $k$, the sets $Z(n-1,k,i)$ and $Z(n-1,k-1,i)$ for all $i$}
\Procedure{FindNextSets}{$n,k$}
\State Input sets $Z(n-1,k,i)$, $Z(n-1,k-1,i)$ for all $i$.
\State $Z(n-1,k) \gets \cup_i \, Z(n-1,k,i)$
\State $Z(n-1,k-1) \gets \cup_i \, Z(n-1,k-1,i)$
\State Descendants $\gets \{\}$
\ForAll{$\Lambda \in Z(n-1,k)$}
\State $M \gets $\textsc{Generators}$(G(\Lambda),X_k)$
\ForAll{$h \in M$}
\If{$h>f(\Lambda)$}
\State Descendants $\gets$ Descendants $\cup \, \{G(\Lambda) \cup \{h\}\}$
\EndIf
\EndFor
\EndFor
\ForAll{$G(\Lambda) \in Z(n-1,k-1)$}
\State $M \gets $\textsc{Generators}$(G(\Lambda),\{g_k\})$
\ForAll{$h \in M$}
\If{$h>f(\Lambda)$}
\State Descendants $\gets$ Descendants $\cup \, \{G(\Lambda) \cup \{h\}\}$
\EndIf
\EndFor
\EndFor
\ForAll{$G(\Lambda) \in $ Descendants}
\State Orbit $\gets \{\}$
\ForAll{$\sigma \in $ Sym$(X_k)$}
\State Orbit $\gets$ Orbit $\cup \, \{\sigma(\Lambda)\}$
\EndFor
\State MinRep $\gets$ minimal representative of Orbit
\State $Z(n,k,|$Orbit$|) \gets Z(n,k,|$Orbit$|) \cup\{$MinRep$\}$
\EndFor
\State Output sets $Z(n,k,i)$ for all $i$.
\EndProcedure
\end{algorithmic}
\end{algorithm}

Using this algorithm, we computed the polynomials and hence the values of $a_n(FS_r)$ for $1 \le n \le 9$ which are presented in Appendix \ref{app:subsemigroup-values}. This was implemented on the Iridis 4 compute cluster \cite{iridis} using C++ code which is available for download \cite{code}. It took 2 hours 30 minutes running on 64 x Intel Xeon E5-2670 processor cores, equivalent to approximately one week of computation on a standard desktop computer.

\section{Ideal growth}

Recall that a subsemigroup $I$ of a semigroup $S$ is called a left (resp. right) ideal of $S$ if $SI \subseteq I$ (resp. $IS \subseteq I$), and a (two-sided) ideal if it is both a left ideal and a right ideal. Let $a_n^{LI}(FS_r)$ denote the number of (Rees) index $n$ left ideals of $FS_r$, let $a_n^{RI}(FS_r)$ denote the number of index $n$ right ideals and $a_n^I(FS_r)$ denote the number of index $n$ two-sided ideals. Note that the number of left ideals is equal to the number of right ideals as $FS_r \to FS_r, w \mapsto rev(w)$ is an anti-isomorphism, and since every ideal is also a subsemigroup it is clear that
\[
a_n^I(FS_r) \le a_n^{LI}(FS_r)=a_n^{RI}(FS_r) \le a_n(FS_r).
\]
We show that both $a_n^{LI}(FS_r)=a_n^{RI}(FS_r)$ and $a_n^I(FS_r)$ have exponential growth, that is, strict growth type $2^n=n^{n/log(n)}$ where strict growth type is defined as in subsemigroup growth.

\subsection{One-sided ideals}

We now make an observation and give an exact formula for the number of index $n$ one-sided ideals of $FS_r$. 

Let $\Lambda$ be an index $n$ right ideal of $FS_r$, then it is clear that given any $w \in G(\Lambda)$, we must also have $w_{\text{pre}(|w|-1)} \in G(\Lambda)$. Therefore considering the right multiplication tree, the set of gaps (including the empty word) corresponds to a rooted $r$-ary tree with $n+1$ vertices. The number of such is precisely the `Fuss-Catalan' numbers:

\[
a_n^{RI}(FS_r)=\frac{1}{(r-1)(n+1)+1}\binom{r(n+1)}{n+1}.
\]

Note that when $r=2$ this reduces to the standard Catalan numbers and when $r=1$, $a_n^{RI}(FS_r)=a_n^{LI}(FS_r)=a_n^{I}(FS_r)=1$ as the only ideals of $(\mathbb{N},+)$ are the ordinary subsemigroups.

\begin{thm} \label{one-sided-growth}
For $r \ge 2$, $FS_r$ has one-sided ideal growth of strict type $2^n=n^{n/log(n)}$.
\end{thm}

\begin{proof}
This follows from the fact that $r^{n+1} \le \binom{r(n+1)}{(n+1)} \le (er)^{n+1}$. 
\end{proof}

For a fixed index $n$, with some basic manipulation the formula above yields a polynomial in $r$ of degree $n$ with no constant term:

\[
a_n^{RI}(FS_r)=\sum_{k=0}^n\left(\left(\frac{s(n+1,k)}{(n+1)!}(n+1)^k\right)r^j\right)
\]
where $s(n,k)$ are the (signed) Stirling numbers of the first kind.
\subsection{Two-sided ideals}

Recalling the subsemigroup tree in Section \ref{sect:uandl} we consider the subtree of all (two-sided) ideals. That this really is a tree follows from the fact that given any ideal $\Lambda \subseteq FS_r$, then $\Lambda^f$ is also an ideal.

\begin{rem} \label{ideal-desc-condition}
Note that given any ideal $\Lambda \subseteq FS_r$ in this subtree, then $\Lambda_h$ is a descendant of $\Lambda$ if and only if $h > f(\Lambda)$ and $h_{\text{pre}(|h|-1)},h_{\text{suf}(|h|-1)} \in G(\Lambda)$.
\end{rem}

The ordinary subsemigroups $O_r(n)$ are clearly ideals. Let $\text{desc}(O_r(n))$ denote the set of ideals that are descendants of $O_r(n)$ in this tree, and let $D(O_r(n))=\{h \in MG(O_r(n)) \mid O_r(n)_h \in \text{desc}(O_r(n))\}$. Clearly the sets $D(O_r(n))$ and $\text{desc}(O_r(n))$ are in bijection. We now construct a lower bound on the size of these sets. First we need a technical lemma.

\begin{lemma} \label{new-desc}
$g_1w_{\text{pre}(|w|-1)} \le w$ for all $w \in FS_r$.
\end{lemma}

\begin{proof}
Let $w=g_{\alpha(1)}\dots g_{\alpha(m)}$ and assume $g_1g_{\alpha(1)}\dots g_{\alpha(m-1)}>g_{\alpha(1)}\dots g_{\alpha(m)}$. Then $g_1 \ge g_{\alpha(1)} \ge g_{\alpha(2)} \ge \dots \ge g_{\alpha(m-1)} \ge g_{\alpha(m)}$ and so $g_{\alpha(i)}=g_1$ for all $1 \le i \le m$ in which case $g_1^m < g_1^m$ which is a contradiction.
\end{proof}

\begin{lemma} \label{non-decreasing-desc}
$|\text{desc}(O_r(n))| \le |\text{desc}(O_r(n+1))|$
\end{lemma}

\begin{proof}
Let $f=f(O_r(n+1))$, then it is clear that whenever $O_r(n)_h$ is an ideal, $O_r(n+1)_h$ is also an ideal for all $h \neq f$. So we need only show that $D(O_r(n+1))$ contains at least one element that $D(O_r(n))$ does not. Let $h=g_1f$, then by Lemma \ref{new-desc}, $h_{\text{pre}(|h|-1)} \le f \in G(O_r(n+1))$ and clearly $h_{\text{suf}(|h|-1)}=f \in G(O_r(n+1)$ but is not in $G(O_r(n))$. Therefore $O_r(n+1)_h$ is an ideal but $O_r(n)_h$ is not an ideal.
\end{proof}

\begin{propn}
$|\text{desc}(O_r(n))| \ge r^m$ where $m=\lfloor log_r((r-1)n+r) \rfloor$.
\end{propn}

\begin{proof}
Consider the case when $n=r+r^2+\dots+r^{m-1}=\frac{r^m-r}{r-1}$ for some $m \ge 1$, then $G(O_r(n))=\{w \in FS_r \mid |w| \le r^{m-1}\}$ and every word $w$ of length $m$ satisfies $w_{\text{pre}(|w|-1)},w_{\text{suf}(|w|-1)} \in G(O_r(n))$ and no longer words do. Hence $|\text{desc}(O_r(n))| = r^m$. Otherwise $n=\frac{r^m-r}{r-1}+i$ for some $m \ge 1$ and some $1 \le i \le r^m-1$. By inductively applying Lemma \ref{non-decreasing-desc}, we have $|\text{desc}(O_r(n))| \ge |\text{desc}(O_r(n-i))| = r^m$.
\end{proof}

This underestimate has the largest error when $n=r+r^2+\dots+r^m-1=\frac{r^{m+1}-r}{r-1}-1$. So take this as a lower bound and for $r \ge 2$, it is always true that
\[
h(n,r):=\frac{r-1}{r}n+\frac{2r-1}{r}=\frac{r-1}{r}(n+1)+1=r^m \le |desc(O_r(n))|.
\]

So if we let

\[
L^I(n,r)=\sum_{i=0}^{K(n,r)}{h(n-i,r)-1 \choose i}
\]
where $K(n,r)=\left\lfloor\frac{r-1}{2r-1}(n+1)\right\rfloor$ is obtained from the inequality $h(n-i,r)-1 \ge i$. Then similar to the argument in Section \ref{sect:uandl} for the lower bound, we have proved the following:

\begin{thm} \label{ideal-lower-bound}
For $r \ge 2$, $a_n^I(FS_r) \ge L^I(n,r)$.
\end{thm}

As a consequence we can see that two-sided ideal growth is bounded below by an expoential:

\begin{thm} \label{ideal-beneath}
For $r \ge 2$, $log(2^n)=O(log(s_n^I(FS_r)))$.
\end{thm}

\begin{proof}

First note that for $r \ge 2$, $n/4 \le K(n,r)$ so by considering the term $i=n/4$ we have
\[
s_n^I(FS_r) \ge a_n^I(FS_r) \ge L^I(n,r) \ge \left(\frac{3n(r-1)/4r+(2r-1)/r}{n/4}\right)^{n/4}  \sim \left(\left(\frac{3(r-1)}{r}\right)^{\frac{1}{4}}\right)^n
\]
Thus $log(2^n)=n=O(log(s_n^I(FS_r)))$.
\end{proof}

Since $a_n^I(FS_r) \le a_n^{RI}(FS_r)$, we immediately deduce from Theorem \ref{one-sided-growth}, $log(s_n^I(FS_r))=O(log(2^n))$ for $r \ge 2$. Hence,

\begin{thm} \label{ideal-growth-strict-type}
For $r \ge 2$, $FS_r$ has ideal growth of strict type $2^n=n^{n/log(n)}$.
\end{thm}

It can be checked that the argument in Section \ref{sect:poly} is valid for ideals also, and so for a fixed index $n$, we also have the following:

\begin{thm} \label{poly-conj-ideal}
$a_n^I(FS_r)$ is a polynomial in $r$ of degree $n$ with no constant term.
\end{thm}

We can easily adapt our previous algorithm to use Remark \ref{ideal-desc-condition} instead of Remark \ref{min-gen-rem}. This is computationally a much easier condition to check. This algorithm was implemented using C++ code which is available for download \cite{code}. The polynomials and hence the values of $a_n^I(FS_r)$ for $1 \le n \le 12$ were calculated and are presented in Appendix \ref{app:ideal-values}. It took less than an hour to calculate running on a single Intel Xeon E5-2670 processor.

We observed an interesting connection between $a_n^I(FS_2)$ and the central binomial coeffecients ${n \choose \lfloor n/2 \rfloor}$ which we are unable to explain (see Appendix \ref{app:cent-binom-coeff}). They agree for the first $6$ values and then the central binomial coefficients seem to be an upper bound. It would be of interest to know whether they are an asymptotic upper bound.

Note that, given any finite index ideal $I \subseteq FS_r$ and any word in the set of gaps $G(I)$, then every prefix and suffix of that word is also in the set of gaps, hence supp$(I) \subseteq G(I)$. This implies that there is only one index $n$ ideal $I$ with $|$supp$(I)|=n$, namely the ordinary subsemigroup, and hence using the notation from Section \ref{sect:poly}, $c(n,n)=1$. Consequently,

\begin{thm}
For a fixed index $n$
\[
a_n^I(FS_r) \sim \frac{n^r}{n!}.
\]
\end{thm}

\section{Congruence growth}

Let Cong$_n(FS_r)$ denote the set of (two-sided) congruences on $FS_r$ with precisely $n$ nonempty congruence classes, and let $a_n^C(FS_r)=|Cong_n(FS_r)|$. We say that a semigroup $S$ is at most $r$ generated if there exists an $r$-element subset of $S$ that generates the semigroup. It is clear that given any at most $r$ generated semigroup $S$ of order $n$, there exists a congruence $\rho \in Cong_n(FS_r)$ such that $FS_r / \rho \cong S$. The number of congruences in $Cong_n(FS_r)$ that quotient to give $S$ is precisely the number of distinct epimorphisms from $FS_r$ to $S$. There are at most $n^r$ such epimorphisms since the map is determined entirely by the image of the generators $X_r$. Hence
\[
f(n,r) \le a_n^C(FS_r) \le n^r \cdot f(n,r)
\]
where $f(n,r)$ is the number of non-isomorphic at most $r$ generated semigroups of order $n$.

\begin{propn} \label{cong-growth-strict-lower-bound}
For $r \ge 2$, $log(2^n)=O(log(s_n^C(FS_r)))$.
\end{propn}

\begin{proof}
Every index $n$ ideal $I$ of $FS_r$ gives rise to a distinct congruence $\rho_I \in $ Cong$_{n+1}(FS_r)$ where $\rho_I=((FS_r \setminus I) \times (FS_r \setminus I)) \cup id_I$ is usually called the Rees congruence on $FS_r$ modulo $I$. Hence $a_n^I(FS_r) \le a_{n+1}^C(FS_r)$ and the result follows by Theorem \ref{ideal-beneath}.
\end{proof}

\begin{thm} \label{cong-strict-growth-type}
$FS_r$ has congruence growth of the same strict type as $f(n,r)$.
\end{thm}

\begin{proof}
For a fixed $r$, it is clear that $f(n,r)$ is non-decreasing and so by the observation above, $s_n^C(FS_r) \le n^{r+1}\cdot f(n,r)$. Therefore $log(s_n^C(FS_r)) \le (r+1) \cdot log(n) + log(f(n,r))$ and so by the previous proposition, $log(s_n(FS_r))=O(log(f(n,r))$. Again, by the observation above $f(n,r) \le a_n^C(FS_r) \le s_n^C(FS_r)$ and so we have $log(s_n^C(FS_r)) \asymp log(f(n,r))$.
\end{proof}

We conclude from this that for $r \ge 2$, the number of non-isomorphic at most $r$-generated semigroups of order $n$ is at least exponential of strict type $2^n=n^{n/log(n)}$ (whereas the number of at most $r$-generated groups of order $n$ is sub-exponential with strict growth type $n^{log(n)}$).

\begin{question}
Can $f(n,r)$ be bounded above by an exponential?
\end{question}

If we can also show that $f(n,r)$ grows at most exponentially then we will have proved that $FS_r$ has congruence growth of strict type $2^n$. In answering the above question, one may be tempted to consider $3$-nilpotent semigroups as almost all finite semigroups are $3$-nilptotent (see \cite{MR0414380} and \cite{MR2946109} for example). However it turns out almost none of the $r$-generated semigroups are $3$-nilpotent, that is, there are no $r$-generated $3$-nilpotent semigroups of order greater than $r^2+r+1$. The {\sf Smallsemi}~\cite{smallsemi} data library in {\sf GAP}
\cite{gap} tells us that for $n$ from $2$ through $8$, $f(n,2)$ equals: $5,17,68,217,670,1937,5686$. This seems to be exponential, and we conjecture that the above question is true.

\subsection{Ascendingly generated tables}

When counting all semigroups of order $n$, a natural question is whether to count order $n$ semigroups up to isomorphism, which we will denote $f(n)$, or whether to count `up to equality' all $n$-element Cayley tables, that is, all binary operations on an $n$-element set, which we will denote $m(n)$. There are at most $n!$ distinct Cayley tables of any semigroup up to isomorphism (precisely when the semigroup has trivial automorphism group). Hence $m(n)=O(n!\cdot f(n))$. As an aside, it is generally believed (although still an open question) that almost all semigroups have trivial automorphism group, in which case not only does $m(n)=O(n! \cdot f(n))$ but $m(n) \sim n! \cdot f(n)$.

We here introduce something between the two which turns out to be important for counting congruences, namely what we have called ascendingly generated tables. Heuristically we are counting the generators $\{w_1,\dots,w_k\}$ up to equality and the non-generators $\{w_{k+1},\dots,w_n\}$ up to isomorphism and it satisfies $O(n^k \cdot f(n,k))$.

Let $W_n=\{w_1 < w_2 < \dots < w_n\}$ be an $n$-element set and let $C_n=\{(W_n,\otimes)\}$ be the set of all Cayley tables on $W_n$. Let $C_{n,k} \subseteq C_n$ be the set of all Cayley tables on $(W_n,\otimes)$ that are generated (not necessarily minimally) by $\{w_1,\ldots,w_k\}$, so that for example, $C_{n,n}=C_n$.

Given any $w \in W_n$, $(W_n,\otimes) \in C_{n,k}$, there exists some decompostion (possibly many) of $w$ as a product of elements from $\{w_1,\dots,w_k\}$. Let dec$(w)$ denote the smallest decomposition with respect to $<_\otimes$, the shortlex order on $W_n$ over the alphabet $\{w_1,\dots,w_k\}$. We say that $(W_n,\otimes)$ is {\bf ascendingly generated} by $\{w_1,\dots,w_k\}$ if it is generated by $\{w_1,\dots,w_k\}$ and
\[
\text{dec}(w_{k+1}) <_{\otimes} \text{dec}(w_{k+2}) <_{\otimes} \dots <_{\otimes} \text{dec}(w_n).
\]
Let $T_{n,k} \subseteq C_{n,k}$ denote the set of all $(W_n,\otimes) \in C_{n,k}$ ascendingly generated by $\{w_1,\dots,w_k\}$.

\begin{lemma} We observe the following facts: \label{ascending-lemma}
\begin{enumerate}
\item $|T_{n,1}|=n$. 
\item $|T_{n,n}|=m(n)$. \label{asc-lemma-equality} 
\item $T_{n,r} \subseteq T_{n,r+1}$. \label{ascending-included}
\item $f(n,r) \le |T_{n,r}| \le r! \cdot {n \choose r} \cdot f(n,r)$. \label{T-n-r-bounds}
\item $|T_{n,r}|=O(n^r \cdot f(n,r))$ for a fixed $r$.
\end{enumerate}
\end{lemma}

\begin{proof}
\begin{enumerate}
\item Let $(W_n,\otimes) \in T_{n,1}$. Since $(W_n,\otimes)$ is ascendingly generated by $\{w_1\}$ we must have
\[
\overbrace{w_1 \otimes \dots \otimes w_1}^i=w_i \text{\qquad for all $1 \le i \le n-1$}
\]
and $(w_i)^n$ is allowed to equal any of the $n$ elements.
\item From the definition, every Cayley table on $W_n$ is generated ascendingly by the whole of $W_n$.
\item Again this follows vacuously from the definition.
\item Given any at most $r$ generated semigroup $S$ we can always label the elements $\{w_1,\dots,w_n\}$ such that $\{w_1,\dots,w_r\}$ generate $S$ and the remaining elements are labelled in the order they are generated so that it is ascendingly generated, hence $f(n,r) \le T_{n,r}$. Given any at most $r$-generated semigroup $S$ there are at most ${n \choose r}$ possible ways of choosing a generating set which we can label in at most $r!$ ways $\{w_1,\dots,w_r\}$, but for the remaining elements we have no choice how to label them if we want $S$ to be generated ascendingly by $\{w_1,\dots,w_r\}$. Hence $|T_{n,r}| \le r! \cdot {n \choose  r} \cdot f(n,r)$.
\item This follows immediately from \ref{T-n-r-bounds}. \qedhere
\end{enumerate}
\end{proof}

Recall that there are $n!$ monogenic Cayley tables of order $n$, but only $n$ up to isomorphism. So ascendingly generated tables behave like `up to isomorphism' for $k=1$ but like `up to equality' for $k=n$.

Using the {\sf Smallsemi}~\cite{smallsemi} data library in {\sf GAP}
\cite{gap} we calculated the number of Cayley tables on $\{w_1,\dots,w_n\}$ generated ascendingly by $\{w_1,\dots,w_k\}$ (see Table 1).

\begin{table}[ht] \center \label{gap-tables}
\begin{tabular}{| c || c | c | c | c | c | c | c | c |}
\hline \bf  k $\backslash$ n  & \bf 1 & \bf 2 & \bf 3 & \bf 4 & \bf 5 & \bf 6 & \bf 7 \\ \hline \hline
\bf 1 &  1 &  2 &  3 &  4 &  5 & 6 & 7 \\ \hline
\bf 2 &   &  8 & 37 & 145 & 452 & 1374 & 3933 \\ \hline
\bf 3 &   &   &  113 &  1257 &  9020 & 60826 & 356023 \\ \hline
\bf 4 &   &   &   &  3492 &  67394 & 938194 & 30492722 \\ \hline
\bf 5 &   &   &   &  &  183732 & 6398792 & 466578957 \\ \hline
\bf 6 &   &   &   &  &  &  17061118 & 3032145644 \\ \hline
\bf 7 &   &   &   &  &  &   &    7743056064 \\ \hline 
\end{tabular}
\caption{$|T_{n,k}|$ for $1 \le k \le n \le 7$.}
\end{table}

Again recall from Lemma \ref{ascending-lemma}(\ref{asc-lemma-equality}) that the diagonal in Table 1 
is equal to $m(n)$, see sequence \href{http://oeis.org/A023814}{A023814} in OEIS \cite{oeis}. 

See \cite{code} for the code used to calculate these values. This ran on the Iridis 4 compute cluster \cite{iridis} and took one hour running on 64 x Intel Xeon E5-2670 processor cores, equivalent to approximately 64 hours on a standard desktop computer.

That ascendingly generated Cayley tables are important is revealed in the next result.

\medskip

A congruence $\rho \in $ Cong$_n(FS_r)$ is completely determined by two pieces of information: an ascendingly generated Cayley table and by which congruence classes the generators $X_r$ are assigned to. We formalise this below.

We say a function $f:X_r \to W_n$ is an {\bf assignment} if $f(g_1)=w_1$ and $f(g_j) \in \{w_1,\dots,w_{\alpha(j)+1}\}$ where $w_{\alpha(j)}=\max_{i < j}\{f(g_i)\}$ for all $2 \le j \le n$. Let $A_{r,k}$ be the set of assignments from $X_{r}$ to $W_{n}$ such that the image has cardinality $k$.

We claim that
\begin{propn} There exists a bijection
\[
\phi:\text{Cong}_n(FS_r) \to \bigsqcup_{1 \le k \le r}(T_{n,k} \times A_{r,k}).
\]
\end{propn}

\begin{proof}
Given any $\rho \in $ Cong$_n(FS_r)$ we choose a subset $W_n \subseteq FS_r$ in the following way: let $w_1 < w_2 < \dots < w_n$ be the minimal representatives of the $n$ $\rho$-classes with respect to $<$, the shortlex order on $FS_r$ over the alphabet $X_r$. Given any word $w \in FS_r$, let $[w]_\rho \in W_n$ denote the minimal representative of the class of $w$. Let $k=|\{[g_1]_\rho,[g_2]_\rho,\dots,[g_r]_\rho\}|$ be the number of classes that the generators $X_r$ traverse.

Let $f$ be defined as $f(g_i):=[g_i]_\rho$ for all $1 \le i \le r$. Clearly $f(g_1)=w_1$ and for any $j>1$, either $g_j$ is in the same $\rho$-class as some $g_i$ for $i < j$ in which case $f(g_j)=f(g_i)$, or it is in a different class from all $g_i$ with $i<j$ in which case $f(g_j)=w_{\alpha(j)+1}$ where $w_{\alpha(j)}=\max_{i<j}\{f(g_i)\}$. Hence $f$ is indeed an assignment and $f \in A_{r,k}$. We now define a binary operation on $W_{n}$, which can be done by setting $w_i \otimes w_j:=[w_iw_j]_\rho$ for all $1 \le i,j \le n$. This is associative as $\rho$ is a congruence. We now show that $(W_n,\otimes)$ is generated by $\{w_1,\dots,w_k\}$. In fact, given any $w \in W_n$, let $w=g_{\alpha(1)}g_{\alpha(2)}\dots g_{\alpha(m)}$ be its unique decomposition in $FS_r$, then $w_i=[w_i]_\rho=[g_{\alpha(1)}g_{\alpha(2)}\dots g_{\alpha(m)}]_\rho=[g_{\alpha(1)}]_\rho\otimes[g_{\alpha(2)}]_\rho\otimes\dots\otimes[g_{\alpha(m)}]_\rho$, where $[g_{\alpha(1)}]_\rho,[g_{\alpha(2)}]_\rho,\dots,[g_{\alpha(m)}]_\rho \in $ im$(f)=\{w_1,\dots,w_k\}$. Hence $(W_n,\otimes) \in C_{n,k}$.

Given any $w_i \in W_n$, let dec$(w_i)=w_{\gamma(1)}\otimes\dots\otimes w_{\gamma(m)}$ be the minimal decomposition of $w_i$. As $w_{i}$ belongs to $FS_{r}$, we also have its unique decomposition $w_i=g_{\alpha(1)}\dots g_{\alpha(l)}$ using letters in $X_{r}$.  We intend to show that $g_{\alpha(1)}\dots g_{\alpha(l)}=w_{\gamma(1)}\dots w_{\gamma(m)}$. To do this, we make two initial observations:
\begin{enumerate}
\item $[g_{\alpha(1)}]_\rho\dots[g_{\alpha(l)}]_\rho\le g_{\alpha(1)}\dots g_{\alpha(l)}$;
\item $[g_{\alpha(1)}]_\rho\dots[g_{\alpha(l)}]_\rho$ is in the same class as $w_i$, which is minimal in its class.
\end{enumerate}
We can immediately deduce that $[g_{\alpha(1)}]_\rho\dots[g_{\alpha(l)}]_\rho=g_{\alpha(1)}\dots g_{\alpha(l)}$ and $g_{\alpha(1)},\dots,g_{\alpha(l)} \in W_n$. As dec$(w_{i})$ is the smallest decomposition of $w_{i}$ in $W_n$ we also have $w_{\gamma(1)}\otimes\dots\otimes w_{\gamma(m)} \le_\otimes g_{\alpha(1)}\otimes\dots\otimes g_{\alpha(l)}$. Assume that $w_{\gamma(1)}\dots w_{\gamma(m)}\neq g_{\alpha(1)}\dots g_{\alpha(l)}$, then $w_{\gamma(1)}\dots w_{\gamma(m)}<w_i$ but in the same $\rho$-class as $w_i$ which is a contradiction. Hence $g_{\alpha(1)}\dots g_{\alpha(l)}=w_{\gamma(1)}\dots w_{\gamma(m)}$ and $w_i < w_j$ if and only if dec$(w_i) <_\otimes $ dec$(w_j)$ thus $(W_n,\otimes) \in T_{n,k}$. Let $\phi(\rho):=((W_n,\otimes),f)$, it is clear that if $\rho=\sigma \in $ Cong$_n(FS_r)$ then $\phi(\rho)=\phi(\sigma)$, whence $\phi$ is a well-defined function.

Now we prove surjectivity of $\phi$ by constructing a congruence: given any $((W_n,\otimes),f) \in \bigsqcup_{1 \le k \le r}(T_{n,k} \times A_{r,k})$ we define $\rho$ as follows: let $(a,b) \in \rho$ if and only if $f(g_{\alpha(1)})\otimes \dots \otimes f(g_{\alpha(l)})=f(g_{\beta(1)})\otimes \dots \otimes f(g_{\beta(m)})$ where $a=g_{\alpha(1)}\dots g_{\alpha(l)}$ and $b=g_{\beta(1)}\dots g_{\beta(m)}$ are their unique decompositions in $FS_r$. It is straightforward to check that this is indeed a congruence. We now show that $\rho$ has $n$ congruence classes. The generators $X_r$ clearly traverse $|$im$(f)|=k$ classes. For each $w_i \in \{w_{k+1},\dots,w_n\}$ let dec$(w_i)=w_{\gamma(1)}\otimes \dots \otimes w_{\gamma(l)}$. Then $w_{\gamma(1)}\dots w_{\gamma(k)}$ is in a different class from all $\{w_1,\dots,w_{i-1}\}$ and therefore $\rho \in $ Cong$_n(FS_r)$. 

Let $\phi(\rho)=((W_n,\oplus),f')$. We intend to show that $((W_n,\otimes),f)=((W_n,\oplus),f')$. Firstly note that $f(g_1)=w_1=f'(g_1)$. We now proceed by induction: given any $j>1$, assume $f(g_i)=f'(g_i)$ for all $i<j$. Then either $f(g_j)=f(g_i)$ for some $i<j$ in which case $f'(g_j)=[g_j]_\rho=[g_i]_\rho=f'(g_i)=f(g_i)=f(g_j)$ or, alternatively, $f(g_j)$ is different from all $f(g_i)$ with $i<j$. In which case $f(g_j)=w_{\alpha(j)+1}$ and $f'(g_j)=w_{\beta(j)+1}$ where $w_{\alpha(j)}=\max_{i<j}\{f(g_i)\}$ and $w_{\beta(j)}=\max_{i<j}\{f'(g_i)\}$. By our assumption, $w_{\alpha(j)}=w_{\beta(j)}$ and so $f=f'$.

We now intend to show that $(W_n,\otimes)=(W_n,\oplus)$. Given any $w_i,w_j \in W_n$, with dec$(w_i)=w_{\gamma(1)}\otimes\dots\otimes w_{\gamma(l)}$, dec$(w_j)=w_{\delta(1)}\otimes\dots\otimes w_{\delta(m)}$, let $w_p=w_i \otimes w_j$ where dec$(w_p)=w_{\epsilon(1)}\otimes \dots \otimes w_{\epsilon(q)}$. By the argument above we know that $w_i=w_{\gamma(1)}\dots w_{\gamma(l)}$, $w_j=w_{\delta(1)}\dots w_{\delta(m)}$, $w_p=w_{\epsilon(1)}\dots w_{\epsilon(q)}$ are also their unique deompositions in $FS_r$ as the minimal representatives of the $\rho$-classes. Hence $w_i \oplus w_j=[w_iw_j]_\rho=[w_p]_\rho=w_p=w_i \otimes w_j$. So $(W_n,\otimes)=(W_n,\oplus)$ and hence $\phi$ is surjective.

Finally, we show that $\phi$ is injective: consider $\phi(\rho)=((W_n,\otimes),f)=\phi(\sigma)$ for some $\rho,\sigma \in $ Cong$_n(FS_r)$. Then $(a,b) \in \rho$ if and only if $f(g_{\alpha(1)})\otimes \dots \otimes f(g_{\alpha(l)})=f(g_{\beta(1)})\otimes \dots \otimes f(g_{\beta(m)})$ where $a=g_{\alpha(1)}\dots g_{\alpha(l)}$ and $b=g_{\beta(1)}\dots g_{\beta(m)}$ if and only if $(a,b) \in \sigma$ and $\phi$ is a bijection.
\end{proof}

Note that $|A_{r,k}|$ is precisely the number of ways of partitioning an $r$-element set in to $k$ non-empty subsets, that is ${r \brace k}$ the Stirling numbers of the second kind. Hence we have proved the following,

\begin{thm}
\[
a_n^C(FS_r)=\sum_{k=1}^r{r \brace k}|T_{n,k}| .
\]
\end{thm}

Hence from the values of $|T_{n,k}|$ in Table \ref{gap-tables} we can calculate $a_n^C(FS_r)$ for $1 \le n \le 7$ which is presented in Appendix \ref{app:congruence-values}.

\begin{cor}
For a fixed rank $r$, $a_n^C(FS_r) \asymp |T_{n,r}|$.
\end{cor}

\begin{proof}
Clearly $a_n^C(FS_r) \ge |T_{n,r}|$. From Lemma \ref{ascending-lemma}(\ref{ascending-included}) $|T_{n,k}| \le |T_{n,r}|$ for all $k \le r$, hence $a_n^C(FS_r) \le B(r) \cdot |T_{n,r}|$ where $B(r)=\sum_{k=1}^r {r \brace k}$ are the Bell numbers.
\end{proof}

Compare this result to Theorem \ref{cong-strict-growth-type} and notice that this is much stronger than saying they have the same strict growth type.

\subsection{Congruence growth for a fixed number of classes}

Given some fixed $n$ we now prove that $a_n^C(FS_r)$ satisfies an exponential equation with base $n$.

\begin{thm}
\[
a_n^C(FS_r)=\sum_{j=1}^n\left(\left(\sum_{k=j}^n(-1)^{k-j}{k \choose j}\frac{|T_{n,k}|}{k!}\right)j^r\right).
\]
\end{thm}

\begin{proof}
\[
a_n^C(FS_r)=\sum_{k=1}^r {r \brace k}|T_{n,k}|=\sum_{k=1}^r\left(\sum_{j=1}^k\frac{(-1)^{k-j}}{k!}{k \choose j}j^r\right)|T_{n,k}|
\]
since $|T_{n,k}|=0$ for all $k>n$,  the result follows by a simple rearrangement.
\end{proof}

Using Table \ref{gap-tables}, we calculate the exponential equations for $1 \le n \le 7$ and present them in Appendix \ref{app:congruence-values}.

\begin{cor}
For a fixed $n$, we have $a_n^C(FS_r) \sim \frac{m(n)}{n!}n^r$.
\end{cor}

\section{Further work}

There are many open questions regarding subsemigroup growth. It was shown in \cite[Theorem 3.1]{MR1978431} that all groups with superexponential subgroup growth are similar to free groups, in that they involve every finite group as an upper section. What necessary conditions are imposed on semigroups? In answering this question, it would be of interest to first investigate other classes of semigroups. What is the subsemigroup growth of free commutative semigroups, free inverse semigroups, the bicyclic monoid etc? For example, if free commutative semigroups have exponential subsemigroup growth, then by an argument similar to Proposition \ref{free-max-growth}, non-commutativity would certainly be a necessary condition for superexponential growth.

How does the geometry of the semigroup relate to its subsemigroup growth, for example, is superexponential growth connected to hyperbolicity?

It may also be that another definition of index is more appropriate. One particularly natural definition that might be considered is Green index \cite{MR2450719} which agrees with Rees index for free semigroups.

It may also be of interest to consider counting other objects, e.g. right/left congruences (related to counting cyclic acts), maximal sub(semi)groups, classes related to Green's relations etc.

\subsection*{Acknowledgments}
The authors acknowledge the use of the IRIDIS High Performance Computing Facility, and associated support services at the University of Southampton, in the completion of this work. The authors would also like to thank Ian Leary and Alistair Wallis for helpful conversations regarding this paper.

\bibliographystyle{plain}
\bibliography{ssgref.bib}

\newpage

\appendix

\begin{landscape}

\section{Values of $a_n(FS_r)$} \label{app:subsemigroup-values}

\bigskip

\begin{tabular}{| c || c | c | c | c | c | c | c | c | c |}
\hline \bf  r $\backslash$ n  & \bf 1 & \bf 2 & \bf 3 & \bf 4 & \bf 5 & \bf 6 & \bf 7 & \bf 8 & \bf 9\\ \hline \hline
\bf 1 &  1 &  2 &  4 &  7 &  12 &  23 &   39 & 67 & 118 \\ \hline
\bf 2 &  2 &  11 &  62 &   382 &  2562 &  18413 &  140968 & 1142004 & 9745298 \\ \hline
\bf 3 &  3 &  27 &  250 &  2568 &  28746 &  347691 &  4495983 & 61714968 & 894242997  \\ \hline
\bf 4 &  4 &  50 &  644 &  9209 &  143416 &  2415078 &  43532832 & 833734416 & 16863679508 \\ \hline
\bf 5 &  5 &  80 &  1320 &  24150 &  480736 &  10340800 &  238120365 & 5826981430 & 150609007570 \\ \hline
\bf 6 &  6 &  117 &  2354 &  52437 &  1269738 &  33192442 &  928558122 & 27600653310 & 866466783828 \\ \hline
\bf 7 &  7 &  161 &  3822 &  100317 &  2859878 &  87935351 &  2892046165 & 101031525714 & 3726895105059 \\ \hline
\bf 8 &  8 &  212 &  5800 &  175238 &  5746592 &  203079088 &  7672012360 & 307755240801 & 13032655134280 \\ \hline
\bf 9 & 9 & 270 & 8364 & 285849 & 10596852 & 423019929 & 18042714315 & 816825050010 & 39027404931886 \\ \hline
\bf 10 & 10 & 335 & 11590 & 442000 & 18274722 & 813079415 & 38632533180 & 1947580054285 & 103592924112830 \\ \hline
\bf 11 & 11 & 407 & 15554 & 654742 & 29866914 & 1465238951 & 76729376515 & 4261622698733 & 249671899238553 \\ \hline
\bf 12 & 12 & 486 & 20332 & 936327 & 46708344 & 2504570454 & 143291607432 & 8692072992879 & 556011110821900 \\ \hline

\end{tabular}

\begin{minipage}{0.3\textwidth}
\begin{align*}
a_1(FS_r)  = &r &\\
a_3(FS_r)  = & \frac{38}{3}r^3-11r^2+\frac{7}{3}r &
\end{align*}
\end{minipage} \hfill
\begin{minipage}{0.7\textwidth}
\begin{align*}
a_2(FS_r)   = & \frac{7}{2}r^2-\frac{3}{2}r & \\
a_4(FS_r)   = & \frac{1201}{24}r^4-\frac{239}{4}r^3+ \frac{311}{24}r^2+\frac{15}{4}r &
\end{align*}
\end{minipage}
\vspace{-0.3cm}\begin{align*}
a_5(FS_r)  = & \frac{6389}{30}r^5-\frac{613}{2}r^4+\frac{185}{6}r^3+\frac{255}{2}r^2-\frac{264}{5}r& \\
a_6(FS_r)  = & \frac{696049}{720}r^6-\frac{72727}{48}r^5-\frac{58627}{144}r^4+\frac{33101}{16}r^3-\frac{509257}{360}r^2+\frac{973}{3}r&\\
a_7(FS_r)  = & \frac{11708603}{2520}r^7-\frac{87143}{12}r^6-\frac{146903}{18}r^5+\frac{54431}{2}r^4-\frac{9126049}{360}r^3+\frac{129725}{12}r^2-\frac{13019}{7}r&\\
a_8(FS_r)  = & \frac{947714177}{40320}r^8-\frac{5336487}{160}r^7-\frac{55786441}{576}r^6+\frac{7419257}{24}r^5-\frac{2105526961}{5760}r^4+\frac{110385341}{480}r^3 -\frac{52875299}{672}r^2+\frac{95103}{8}r &\\
a_9(FS_r)  = &\frac{5649947729}{45360}r^9-\frac{78967849}{560}r^8-\frac{1039050691}{1080}r^7+\frac{142822454}{45}r^6-\frac{9770306269}{2160}r^5+\frac{2708660903}{720}r^4-\frac{44177206909}{22680}r^3 \\
&+\frac{378138079}{630}r^2-\frac{776555}{9}r&
\end{align*}

\section{Values of $a_n^I(FS_r)$} \label{app:ideal-values}

\bigskip

\begin{tabular}{| c || c | c | c | c | c | c | c | c | c | c | c | c |}
\hline \bf  r $\backslash$ n  & \bf 1 & \bf 2 & \bf 3 & \bf 4 & \bf 5 & \bf 6 & \bf 7 & \bf 8 & \bf 9 & \bf 10 & \bf 11 & \bf 12 \\ \hline \hline
\bf 1 & 1 & 1 & 1 & 1 & 1 & 1 & 1 & 1 & 1 & 1 & 1 & 1  \\ \hline
\bf 2 & 2 & 3 & 6 & 10 & 20 & 35 & 68 & 126 & 242 & 458 & 886 & 1696  \\ \hline
\bf 3 & 3 & 6 & 16 & 36 & 96 & 237 & 624 & 1608 & 4221 & 11043 & 29109 & 76768  \\ \hline
\bf 4 & 4 & 10 & 32 & 89 & 284 & 866 & 2776 & 8860 & 28744 & 93464 & 305608 & 1000982  \\ \hline
\bf 5 & 5 & 15 & 55 & 180 & 656 & 2330 & 8620 & 32020 & 120900 & 459660 & 1761230 & 6779350  \\ \hline
\bf 6 & 6 & 21 & 86 & 321 & 1302 & 5212 & 21582 & 90132 & 382602 & 1639917 & 7096674 & 30926564  \\ \hline
\bf 7 & 7 & 28 & 126 & 525 & 2331 & 10297 & 46796 & 215012 & 1003877 & 4740008 & 22622985 & 108914792  \\ \hline
\bf 8 & 8 & 36 & 176 & 806 & 3872 & 18600 & 91520 & 455849 & 2306152 & 11808484 & 61161312 & 319883860  \\ \hline
\bf 9 & 9 & 45 & 237 & 1179 & 6075 & 31395 & 165591 & 884592 & 4796848 & 26337348 & 146326572 & 821478540  \\ \hline
\bf 10 & 10 & 55 & 310 & 1660 & 9112 & 50245 & 281920 & 1602175 & 9236660 & 53921531 & 318568940 & 1902539090  \\ \hline

\end{tabular}

{ \footnotesize
\hspace{0.55cm}\begin{minipage}{0.1\textwidth}
\[
a_1^I(FS_r)=r
\]
\end{minipage} \hspace{5cm}
\begin{minipage}{0.2\textwidth}
\[
a_2^I(FS_r)=\frac{1}{2}r^2+\frac{1}{2}r
\]
\end{minipage} \hfill
\begin{minipage}{0.65\textwidth}
\[
a_3^I(FS_r)=\frac{1}{6}r^3+\frac{3}{2}r^2-\frac{2}{3}r
\]
\end{minipage}

\hspace{0.55cm}\begin{minipage}[b]{0.3\textwidth}
\[
a_4^I(FS_r)=\frac{1}{24}r^4+\frac{5}{4}r^3-\frac{1}{24}r^2-\frac{1}{4}r
\]
\end{minipage} \hfill
\begin{minipage}[b]{0.82\textwidth}
\[
a_5^I(FS_r)=\frac{1}{120}r^5+\frac{7}{12}r^4+\frac{67}{24}r^3-\frac{43}{12}r^2+\frac{6}{5}r
\]
\end{minipage}

\begin{align*}
a_6^I(FS_r)=&\frac{1}{720}r^6+\frac{3}{16}r^5+\frac{461}{144}r^4-\frac{73}{48}r^3-\frac{1513}{360}r^2+\frac{10}{3}r \\
a_7^I(FS_r)=&\frac{1}{5040}r^7+\frac{11}{240}r^6+\frac{263}{144}r^5+\frac{115}{16}r^4-\frac{8089}{360}r^3-\frac{319}{-15}r^2-\frac{48}{7}r \\
a_8^I(FS_r)=&\frac{1}{40320}r^8+\frac{13}{1440}r^7+\frac{1979}{2880}r^6+\frac{173}{18}r^5-\frac{66113}{5760}r^4-\frac{43913}{1440}r^3+\frac{227777}{3360}r^2+\frac{281}{-8}r \\
a_9^I(FS_r)=&\frac{1}{362880}r^9+\frac{1}{672}r^8+\frac{1657}{8640}r^7+\frac{547}{90}r^6+\frac{377749}{17280}r^5-\frac{37105}{288}r^4+\frac{18446699}{90720}r^3-\frac{294191}{2520}r^2+\frac{136}{9}r \\
a_{10}^I(FS_r)=&\frac{1}{3628800}r^{10}+\frac{17}{80640}r^9+\frac{5129}{120960}r^8+\frac{14423}{5760}r^7+\frac{5648053}{172800}r^6-\frac{789689}{11520}r^5-\frac{20055283}{90720}r^4+\frac{18449327}{20160}r^3-\frac{28177631}{25200}r^2+\frac{2292}{5}r \hfill & \\
a_{11}^I(FS_r)=&\frac{1}{39916800}r^{11}+\frac{19}{725760}r^{10}+\frac{937}{120960}r^9+\frac{91897}{120960}r^8+\frac{3771383}{172800}r^7+\frac{2583703}{34560}r^6-\frac{133247833}{181440}r^5  \\
&+\frac{289546877}{181440}r^4-\frac{5427659}{6300}r^3-\frac{1158377}{1260}r^2+\frac{9054}{11}r\\
a_{12}^I(FS_r)=&\frac{1}{479001600}r^{12}+\frac{1}{345600}r^{11}+\frac{10487}{8709120}r^{10}+\frac{17567}{96768}r^9+\frac{137379751}{14515200}r^8+\frac{1567309}{12800}r^7-\frac{664752743}{1741824}r^6 \\
&-\frac{22482445}{13824}r^5+\frac{108858294689}{10886400}r^4-\frac{11855920577}{604800}r^3+\frac{5629166951}{332640}r^2-\frac{21769}{4}r
\end{align*} }

\section{Values of $a_n^C(FS_r)$} \label{app:congruence-values}

\label{cong-tables} \begin{tabular}{| c || c | c | c | c | c | c | c | c | c |}
\hline \bf  r $\backslash$ n  & \bf 1 & \bf 2 & \bf 3 & \bf 4 & \bf 5 & \bf 6 & \bf 7 \\ \hline \hline
\bf 1 & 1 & 2 & 3 & 4 & 5 & 6 & 7 \\ \hline
\bf 2 & 1 & 10 & 40 & 149 & 457 & 1380 & 3940 \\ \hline
\bf 3 & 1 & 26 & 227 & 1696 & 10381 & 64954 & 367829 \\ \hline
\bf 4 & 1 & 58 & 940 & 12053 & 124683 & 1312774 & 32656398 \\ \hline
\bf 5 & 1 & 122 & 3383 & 68524 & 1089957 & 17321988 & 780465754 \\ \hline
\bf 6 & 1 & 250 & 11320 & 344609 & 7962407 & 179542398 & 12045020929 \\ \hline
\bf 7 & 1 & 506 & 36347 & 1609696 & 52053881 & 1600876052 & 147519031977 \\ \hline
\bf 8 & 1 & 1018 & 113860 & 7172573 & 316326523 & 12911778902 & 1565476753784 \\ \hline
\bf 9 & 1 & 2042 & 351263 & 30972244 & 1828173277 & 97095768316 & 15081546028136 \\ \hline
\bf 10 & 1 & 4090 & 1073200 & 130896569 & 10196063247 & 694127660206 & 135628506406503 \\ \hline
\end{tabular}

\bigskip

\begin{eqnarray*}
a_1^C(FS_r) & = &1\\
a_2^C(FS_r) & = & 4\cdot2^r-6 \\
a_3^C(FS_r) & = & \frac{113}{6}3^r-38\cdot2^r+\frac{45}{2} \\
a_4^C(FS_r) & = & \frac{291}{2}4^r-\frac{745}{2}3^r+ 317\cdot 2^r-\frac{189}{2} \\
a_5^C(FS_r) & = & \frac{15311}{10}5^r-\frac{58169}{12}4^r+5582\cdot 3^r-\frac{5493}{2}2^r+\frac{2917}{6} \\
a_6^C(FS_r) & = & \frac{8530559}{360}6^r-\frac{444264}{5}5^r+\frac{3069971}{24}4^r-\frac{782245}{9}3^r+\frac{216245}{8}2^r-\frac{43211}{15} \\
a_7^C(FS_r) & = & \frac{161313668}{105}7^r-\frac{235545521}{36}6^r+\frac{261192269}{24}5^r-\frac{210522757}{24}4^r+\frac{122535049}{36}3^r-\frac{7912154}{15}2^r+\frac{458861}{24}
\end{eqnarray*}

\end{landscape}

\section{Comparing $a_n^I(FS_2)$ with central binomial coeffecients} \label{app:cent-binom-coeff}

\center

\bigskip

\hspace{-3em} \label{bin-coeff} \begin{tabular}{| c || c | c | c | c |}
\hline & & & \\
\bf  $n$  & \bf $a_n^I(FS_2)$ & \bf ${n \choose \lfloor n/2 \rfloor}$ & Difference  \\[10pt] \hline \hline
\bf 1 & 2 & 2 & 0 \\ \hline
\bf 2 & 3 & 3 & 0 \\ \hline
\bf 3 & 6 & 6 & 0 \\ \hline
\bf 4 & 10 & 10 & 0 \\ \hline
\bf 5 & 20 & 20 & 0 \\ \hline
\bf 6 & 35 & 35 & 0 \\ \hline
\bf 7 & 68 & 70 & 2 \\ \hline
\bf 8 & 126 & 126 & 0 \\ \hline
\bf 9 & 242 & 252 & 10 \\ \hline
\bf 10 & 458 & 462 & 4 \\ \hline
\bf 11 & 886 & 924 & 38 \\ \hline
\bf 12 & 1696 & 1716 & 20 \\ \hline
\bf 13 & 3284 & 3432 & 148 \\ \hline
\bf 14 & 6339 & 6435 & 96 \\ \hline
\bf 15 & 12302 & 12870 & 568 \\ \hline
\bf 16 & 23850 & 24310 & 460 \\ \hline
\bf 17 & 46390 & 48620 & 2230 \\ \hline
\bf 18 & 90244 & 92378 & 2134 \\ \hline
\bf 19 & 175940 & 184756 & 8816 \\ \hline
\bf 20 & 343246 & 352716 & 9470 \\ \hline
\bf 21 & 670714 & 705432 & 34718 \\ \hline
\bf 22 & 1311764 & 1352078 & 40314 \\ \hline
\bf 23 & 2568740 & 2704156 & 135416 \\ \hline
\bf 24 & 5034652 & 5200300 & 165648 \\ \hline
\bf 25 & 9877768 & 10400600 & 522832 \\ \hline
\end{tabular}

\end{document}